\documentclass{amsart}
\usepackage[english]{babel}
\usepackage{amsmath, amsthm, amssymb}
\usepackage{enumitem}
\usepackage{mathrsfs}
\usepackage{import}
\usepackage{scalerel}
\usepackage[hidelinks]{hyperref}

\title[Countable compactness in powers of topological groups]
{Countable compactness in powers of topological groups and Ramsey theoretic variations of compactness}

\author[V. O. Rodrigues]{Vinicius de Oliveira Rodrigues}
\address{Institute of Mathematics, Statistics and Computer Science, University of São Paulo, Rua do Matão 1010, 05508-090 São Paulo, SP, Brazil}
\email{vinior@ime.usp.br}
\thanks{The first author was supported in part by the São Paulo Research Foundation (FAPESP), Brazil, grant 2025/07302-0.}

\author[P. J. Szeptycki]{Paul Jan Szeptycki}
\address{Department of Mathematics and Statistics, York University, 4700 Keele Street, Toronto, Ontario M3J 1P3, Canada}
\email{szeptyck@yorku.ca}

\author[A. H. Tomita]{Artur Hideyuki Tomita}
\address{Institute of Mathematics, Statistics and Computer Science, University of São Paulo, Rua do Matão 1010, 05508-090 São Paulo, SP, Brazil}
\email{tomita@ime.usp.br}


\newtheorem{theorem}{Theorem}[section]
\newtheorem{proposition}[theorem]{Proposition}
\newtheorem{lemma}[theorem]{Lemma}
\newtheorem{corollary}[theorem]{Corollary}
\theoremstyle{definition}
\newtheorem{definition}[theorem]{Definition}

\theoremstyle{plain}
\newtheorem*{claim}{Claim}
\newtheorem*{lemma*}{Lemma}
\newtheorem*{theorem*}{Theorem}

\DeclareMathOperator{\cl}{cl}
\DeclareMathOperator{\cf}{cf}
\DeclareMathOperator{\Fun}{Fun}
\DeclareMathOperator{\ran}{ran}
\DeclareMathOperator{\dom}{dom}
\DeclareMathOperator{\finct}{FinCt}
\DeclareMathOperator{\li}{LI^{*}}
\DeclareMathOperator{\LI}{LI^{+}}
\let\bigopsize\bigoplus
\def\bigominus{{\scalerel*{\boldsymbol\ominus}{\bigopsize}}}

\begin{document}
\begin{abstract}
Countable compactness need not be preserved by finite products. Motivated by
compactness notions for maps indexed by finite subsets of $\omega$, we introduce
cascade countable compactness for arbitrary barriers. For
$\mathcal B=[\omega]^2$, this is the previously studied notion of being doubly
countably compact.
We show that, in ZFC, if $G$ is a Hausdorff topological
group and $1\leq k<\omega$, then $k$-cascade countable compactness of $G$
implies that $G^k$ is countably compact. Cascade countable compactness for the
Schreier barrier implies that $G^\omega$ is countably compact. In contrast, we
construct a Tychonoff space that is $n$-cascade countably compact for every
$n<\omega$ but has a non-countably compact square, and a Hausdorff Boolean
group $H$ that is $\mathcal B$-countably compact for every barrier $\mathcal B$
but whose square is not countably compact. We also construct a Hausdorff
Boolean group without nontrivial convergent sequences that is
$\mathcal B$-cascade countably compact for every barrier $\mathcal B$. Finally,
we obtain a subspace $X\subseteq\beta\omega$ such that $X^\kappa$ is
$n$-cascade countably compact for every $\kappa<\mathfrak h$ and every
$n<\omega$, whereas $\exp X$ is not pseudocompact.
\end{abstract}
    \maketitle
    \section{Introduction}\label{section:introduction}

Tychonoff's Theorem states that an arbitrary product of compact spaces is
compact. For weaker compactness properties, product preservation may fail, and
one may ask which additional hypotheses imply it.

Two central weakenings of compactness are pseudocompactness and countable
compactness.
A topological space is pseudocompact if every continuous real-valued function
on it is bounded.
It is countably compact if every sequence has an accumulation
point.
Every countably compact space is pseudocompact.

Terasaka and Nov\'ak provided examples of pairs of countably compact spaces
whose product is not even pseudocompact
\cite{terasaka1952cartesian,novak1953cartesian}.
Thus, neither pseudocompactness nor countable compactness is preserved even under finite products of arbitrary
spaces.
For topological groups, by contrast, Comfort and Ross proved that an arbitrary product of pseudocompact
topological groups is pseudocompact \cite{comfort1966pseudocompactness}.

The corresponding problem for countable compactness became one of the
longest-standing questions in the theory of topological groups.
In 1966, Comfort asked whether there are countably compact groups $G_0$ and
$G_1$ for which
$G_0\times G_1$ is not countably compact; the question was later recorded in
\cite{comfort1990problems}.

The first results were obtained under additional set-theoretic
hypotheses.
Hajnal and Juh\'asz constructed under the Continuum Hypothesis a countably
compact group without nontrivial convergent sequences
\cite{hajnal1976separable}.
In 1980, van Douwen gave the first consistent answer to Comfort's product
question: under Martin's Axiom, he constructed two countably compact groups
whose product is not countably compact \cite{vandouwen1980product}.
His paper also proved in ZFC that every countably compact Boolean group without
nontrivial convergent sequences contains two countably compact subgroups
whose product is not countably compact. He asked whether such a Boolean group
could be constructed in ZFC, a question known as van Douwen's problem. A
positive answer would also give a ZFC example for Comfort's question.
Hart and van Mill strengthened the consistent counterexample in 1991: under
$\mathrm{MA}_{\mathrm{countable}}$, they produced a Boolean countably compact
group $H$ whose square $H^2$ is not countably compact
\cite{hart1991countably}.
These results gave consistent examples, while the ZFC problem remained open.

Tomita strengthened van Douwen's reduction by proving in ZFC that the
existence of a countably compact Abelian group without non-trivial convergent
sequences implies the existence of a countably compact group whose square is
not countably compact \cite{tomita2005square}.
Consistent constructions also exhibited a wide range of behavior in powers.
Assuming the existence of $2^{\mathfrak c}$ selective ultrafilters and
$2^{\mathfrak c}=2^{<2^{\mathfrak c}}$, Tomita proved that, for every cardinal
$\alpha\leq 2^{\mathfrak c}$, there is a topological group $G$ such that
$G^\gamma$ is countably compact for every $\gamma<\alpha$, whereas $G^\alpha$
is not \cite{tomita2005powers}.
For hereditarily finally dense (HFD) groups, Szeptycki and Tomita showed under
the Continuum Hypothesis that, for every positive integer $n$, there is an HFD
group $G$ such that $G^n$ is countably compact but $G^{n+1}$ is not. They also
proved that the countable power of each HFD example constructed using random
reals is countably compact \cite{szeptycki2014powers}.
Tomita later showed that, if $\kappa\leq\omega$ and an Abelian topological group
$H$ without non-trivial convergent sequences has $H^n$ countably compact for
every $n<\kappa$, then there is a group $G$ such that $G^n$ is countably compact
for every $n<\kappa$, but $G^\kappa$ is not. If $H$ is torsion, the same transfer
holds for $\kappa=\omega_1$ \cite{tomita2019vandouwen}.

Finally, in 2021, Hru\v{s}\'ak, van Mill, Ramos-Garc\'\i a, and Shelah later constructed in ZFC a
countably compact Boolean subgroup of $2^{\mathfrak c}$ without non-trivial
convergent sequences \cite{hruvsak2021countably}.
This answered both van Douwen's problem and, through his reduction, also Comfort's a question in ZFC more than half a century after it
was posed.

Ultrafilter limits are a useful tool for exploring the structure and productivity of countably compact spaces.
Given a topological space $X$ and a free ultrafilter $p$ on $\omega$, a sequence $(a_n:n \in \omega)$ and $a \in X$, we say that $a$ is a $p$-limit of $(a_n:n \in \omega)$ if $\{n\in \omega:a_n\in U\}\in p$ for every neighborhood
$U$ of $x$.
In this case, the sets of the form $\{n\in \omega:a_n\in U\}$, where $U$ is a neighborhood of $x$, form a filter base for $p$.
The fact that, in general, countable compactness is not preserved by products is due to the fact that given two sequences $(a_n:n\in\omega)$ and $(b_n:n\in\omega)$ and respective accumulation points $a$ and $b$, there may exist open neighborhoods $U$ of $a$ and $V$ of $b$ such that $\{n\in\omega: a_n\in U\}\cap \{n\in\omega: b_n\in V\}$ have empty intersection.
Consequently, there is no free ultrafilter $p$ on $\omega$ such that both sequences have $p$-limits.

In a compact space, every sequence has a $p$-limit with respect to every free ultrafilter $p$ on $\omega$.
All $p$-limits are accumulation points of their respective sequences, so if every sequence has a $p$-limit (where $p$ may depend on the sequence), then the space is countably compact.
In fact, a space $X$ is countably compact if and only if every sequence has a $p$-limit for some free ultrafilter $p$ on $\omega$.

Bernstein introduced an intermediate property between compactness and
countable compactness \cite{bernstein1970new}. Given a free ultrafilter $p$, a
space is $p$-compact if every sequence has a $p$-limit.
Every compact space is $p$-compact for every free ultrafilter $p$, and every $p$-compact space is countably compact.
Moreover, since $p$-limits work coordinatewise, the notion of $p$-compactness is preserved by arbitrary products.
This motivates the following question:

\begin{quote}
Which conditions weaker than $p$-compactness guarantee that, for each sequence
in a product of topological spaces or groups, there exists a free ultrafilter $q$ for which all coordinate sequences
have $q$-limits?
\end{quote}

Barriers are combinatorial objects that extend the role played by the families $[\omega]^n$ of all $n$-element subsets of $\omega$.
Many classical Ramsey-theoretic phenomena that hold for $[\omega]^n$ admit natural analogues in the context of barriers, as illustrated by the Nash--Williams Theorem \cite{nash1965well}, which extends the classical infinite Ramsey theorem to barriers.
Barriers come equipped with ranks taking values in the countable ordinals, where the rank of $[\omega]^n$ is $n$.
Thus, they provide a natural framework for defining higher-order versions of several topological properties related to products.

This perspective has recently been developed in the context of sequential compactness and countable compactness.
The barrier version of sequential compactness was introduced in \cite{kubis2023topological} and further studied in \cite{corral2024infinite}.
Corral, Memarpanahi, and Szeptycki introduced in \cite{corral2024highCompactness} barrier versions of countable compactness associated with fronts and barriers.
Two notions are especially relevant for the present paper.
The first is $\mathcal B$-countable compactness, introduced in \cite{corral2024highCompactness}, and defined by requiring the existence of filter-limit points for functions $f:\mathcal B\to X$.
The second is the stronger notion of $\mathcal B$-cascade countable compactness.
It was introduced in the two-dimensional case as doubly countably compact in \cite{banakh2009rees} and is generalized in the present paper.

Our first product result, Theorem~\ref{theorem:productOfCountablyCompactGroups}, shows that, for every Hausdorff topological group $G$ and every $1\leq k<\omega$,
\[
    G\text{ is $k$-cascade countably compact}
    \quad\Longrightarrow\quad
    G^k\text{ is countably compact}.
\]
For the Schreier barrier, Corollary~\ref{corollary:productOfCountablyCompactGroups} gives
\[
    G\text{ is Schreier-cascade countably compact}
    \quad\Longrightarrow\quad
    G^\omega\text{ is countably compact}.
\]
These implications do not hold for arbitrary Tychonoff spaces: we construct a space that is
$n$-cascade countably compact for every finite $n$, while its square is not
countably compact; see
Corollary~\ref{corollary:betaOmegaHighDimensionalProduct}.

The cascade hypothesis in the group product theorem cannot be replaced by
ordinary barrier countable compactness. We construct, in ZFC, a Hausdorff Boolean group $H$ with no nontrivial
convergent sequences that is $\mathcal B$-countably compact for every barrier
$\mathcal B$, while $H^2$ is not countably compact; see
Theorem~\ref{theorem:exampleOfBCountablyCompactGroup}.
We also construct a Hausdorff Boolean group $G$ of cardinality $\mathfrak c$,
with no nontrivial convergent sequences, that is $\mathcal B$-cascade countably
compact for every barrier $\mathcal B$. By the Schreier case of our product
theorem, $G^\omega$ is countably compact; see
Theorem~\ref{theorem:exampleOfBCascadeCountablyCompactGroup}.

Our results for general spaces also answer several open questions about these
compactness notions. Assuming a Ramsey ultrafilter, Corral, Memarpanahi, and
Szeptycki \cite{corral2024highCompactness} constructed a doubly countably compact
subspace of $\beta\omega$ with non-countably compact square and a
$2$-countably compact space that is not doubly countably compact.
Questions~4.15 and~4.16 of that paper ask whether these examples exist in ZFC, and Question~4.17 asks for the relevant two-dimensional compactness of complements of subsets of $\omega^*$ of cardinality $\mathfrak c$.
We answer all three questions in ZFC and extend the corresponding conclusions
to every finite dimension.
More precisely, if $A\subseteq\beta\omega$ and $|A|<2^{\mathfrak c}$, then
\[
    \beta\omega\setminus A
    \text{ is $n$-cascade countably compact for every }n<\omega;
\]
see Proposition~\ref{proposition:betaOmegaSmallDimensional}.
From this we obtain the square counterexample above and a space that is $k$-countably compact for every $k<\omega$ but not $2$-cascade countably compact (that is, doubly countably compact); see Proposition~\ref{proposition:betaOmega2NotCascade}.

The proofs use a theorem of Carlson, Hindman, and Strauss: for every map
$f:[\omega]^n\to X$ into a Hausdorff space, there is an infinite
$M\subseteq\omega$ such that $f[[M]^n]$ is discrete
\cite{carlson2005discrete}. We show that the analogous statement fails for the
Schreier barrier and, more generally, for uniform barriers of infinite rank.
Accordingly, the discrete-thinning argument applies only to finite ranks.

Finally, we consider the Vietoris hyperspace $\exp X$. Ginsburg proved that if
every power of a Tychonoff space $X$ is countably compact, then $\exp X$ is
countably compact, and asked whether $\exp X$ is pseudocompact whenever every
power of $X$ is pseudocompact \cite{ginsburg1975some}. The latter question was
answered negatively in \cite{hruvsak2007pseudocompactness}. This result was
strengthened in \cite{ortiz2018small} by constructing a subspace
$X\subseteq\beta\omega$ such that $X^\kappa$ is countably compact for every
$\kappa<\mathfrak h$, whereas $\exp X$ is not pseudocompact. We construct such
an $X$ for which $X^\kappa$ is, in addition, $n$-cascade countably compact for
every $\kappa<\mathfrak h$ and every $n<\omega$; see
Corollary~\ref{corollary:mainHyperspaceH}.

The paper is organized as follows.
Section~\ref{section:preliminaries} collects notation, barrier combinatorics, and the two compactness notions.
Section~\ref{section:discrete-images} proves the failure discrete-thinning on uniform barriers of infinite rank.
Section~\ref{section:betaomega} gives the ZFC constructions in $\beta\omega$ and answers Questions~4.15--4.17 of \cite{corral2024highCompactness}.
Section~\ref{section:product-groups} proves the finite and countable product theorems for groups.
Section~\ref{section:boolean-groups} gives the two Boolean group constructions described above.
Section~\ref{section:hyperspaces} treats Vietoris hyperspaces.

\section{Preliminaries}\label{section:preliminaries}
\subsection{Basic Notation}
We use standard set-theoretic notation.
For convenience, we record here the conventions that will be used throughout the paper.

Given a set $X$, the power set of $X$, denoted $\mathcal{P}(X)$, is the collection of all subsets of $X$.
If $\kappa$ is a cardinal, then $[X]^{\kappa}$ is the collection of all subsets of $X$ of cardinality $\kappa$.
Likewise, $[X]^{<\kappa}$ is the collection of all subsets of $X$ of cardinality less than $\kappa$, and $[X]^{\leq \kappa}$ is the collection of all subsets of $X$ of cardinality at most $\kappa$.

If $a, b \subseteq \omega$, we say that $a$ is an \emph{initial segment} of $b$ ($a\sqsubseteq b$) if $a\subseteq b$ and whenever $n,m \in b$ satisfy $n<m$ and $m \in a$, then $n \in a$.
If $a\neq b$, we say that $a$ is a \emph{proper initial segment} of $b$ ($a\sqsubset b$) if $a\sqsubseteq b$ and $a\neq b$.
We write $a<b$ if for every $n \in a$ and $m \in b$ we have $n<m$.

If $A\subseteq \omega$ and $n \in \omega$, we define
\[
    A\restriction_{>n}=\{m \in A: m>n\}.
\]
For convenience, if $s \in [\omega]^{<\omega}$, we write $A\restriction_s=A\restriction_{>\max s}$ if $s\neq \emptyset$, and $A\restriction_s=A$ if $s=\emptyset$.

\begin{definition}
    Let $\mathcal F\subseteq [\omega]^{<\omega}$, and let $M\subseteq\omega$.
    The restriction of $\mathcal F$ to $M$ is
    \begin{equation*}
        \mathcal{F}\restriction M = \{s \in \mathcal{F} : s \subseteq M\}=\mathcal F\cap[M]^{<\omega}.
    \end{equation*}

    The \emph{$\sqsubseteq$-tree generated by $\mathcal F$} is
    \begin{equation*}
        \mathcal{T}_{\mathcal{F}}=\{s \in [\omega]^{<\omega} : \exists t \in \mathcal{F}\ (s\sqsubseteq t)\}.
    \end{equation*}
    We order $\mathcal{T}_{\mathcal{F}}$ by $\sqsubseteq$.
\end{definition}
Note that, in general, $\mathcal F\restriction M$ is \emph{not} the set $\{s\cap M: s\in \mathcal F\}$.
Moreover, $\mathcal{T}_{\mathcal{F}}$ is a tree: for every $t\in \mathcal{T}_{\mathcal{F}}$, the set of predecessors of $t$ in $\mathcal{T}_{\mathcal{F}}$ is finite and linearly ordered by $\sqsubseteq$.

\begin{definition}
    Let $\mathcal{T}\subseteq [\omega]^{<\omega}$.
    For every $s \in [\omega]^{<\omega}$, we define the collection of extensions of $s$ in $\mathcal T$, denoted $\mathcal{T}_s$, by
    \begin{equation*}
        \mathcal{T}_s = \{t\setminus s: s \sqsubseteq t \in \mathcal T\}.
    \end{equation*}
    We regard $\mathcal T_s$ as ordered by $\sqsubseteq$.
\end{definition}

If $\mathcal T$ is a tree, then $\mathcal T_s$ is also a tree.

\subsection{Basics on barriers}
Fronts and barriers are higher-order analogs of the families $[\omega]^n$.
In this subsection, we review their basic properties.

\begin{definition}
Let $\mathcal{F}$ be a collection of finite subsets of $\omega$. We say that:
\begin{enumerate}[label=(\alph*)]
    \item $\mathcal{F}$ is \emph{Ramsey} if for every finite partition $\mathscr{G}$ of $\mathcal F$ and for every infinite set $N \subseteq \omega$, there exists an infinite set $M \subseteq N$ and $\mathcal G \in \mathscr{G}$ such that $\mathcal{F}\restriction M=\mathcal G\restriction M$.
    \item $\mathcal{F}$ is \emph{Nash-Williams}, or \emph{thin}, if $\mathcal F$ is a $\sqsubseteq$-antichain, that is: for every two distinct $s, t \in \mathcal F$, $s\not \sqsubseteq t$ and $t\not \sqsubseteq s$.
    \item $\mathcal{F}$ is \emph{Sperner} if $\mathcal F$ is a $\subseteq$-antichain, that is: for every two distinct $s, t \in \mathcal F$, $s\not \subseteq t$ and $t\not \subseteq s$.
\end{enumerate}
\end{definition}

It is easy to verify that all three classes of collections are closed under taking restrictions, that is, if $\mathcal{F}$ is Ramsey, Nash-Williams or Sperner, then so is $\mathcal{F}\restriction M$ for every infinite $M \subseteq \omega$.

The following is a celebrated Ramsey-type result for Nash-Williams collections.
For a modern proof, we also refer to \cite[Theorem 1.14]{todorcevic2010introduction}.
\begin{theorem}[Nash-Williams \cite{nash1965well}]
    Every Nash-Williams collection is Ramsey.
\end{theorem}

We now introduce fronts and barriers.
\begin{definition}
    Let $M \subseteq \omega$ be infinite and $\mathcal F \subseteq [M]^{<\omega}$.
    \begin{enumerate}[label=(\alph*)]
        \item $\mathcal{F}$ is a \emph{front} on $M$ if $\mathcal{F}$ is Nash-Williams and, for every infinite $N \subseteq M$, there exists $s \in \mathcal F$ such that $s\sqsubseteq N$.
        \item $\mathcal{F}$ is a \emph{barrier} on $M$ if $\mathcal{F}$ is a Sperner front on $M$.
    \end{enumerate}
\end{definition}

\begin{proposition}
    Let $M \in [\omega]^{\omega}$ and $\mathcal F \subseteq [M]^{<\omega}$ be a front on $M$.
    Then the relation $\sqsupset$ is well-founded on $\mathcal T_{\mathcal F}$; equivalently, $\mathcal T_{\mathcal F}$ has no infinite $\sqsubseteq$-branches.
\end{proposition}
We define the rank of a front as follows:
\begin{definition}
    Let $M \in [\omega]^{\omega}$ and $\mathcal F \subseteq [M]^{<\omega}$ be a front on $M$.
    The \emph{rank} of $s \in \mathcal T_{\mathcal F}$, denoted by $\mathrm{rk}_{\mathcal F}(s)$, is recursively defined by
    \[
        \mathrm{rk}_{\mathcal F}(s)
        =
        \sup\{\mathrm{rk}_{\mathcal F}(s\cup\{n\})+1 : n \in M\restriction_s,\ s\cup\{n\}\in \mathcal T_{\mathcal F}\}.
    \]
    The rank of $\mathcal F$, denoted by $\mathrm{rk}(\mathcal F)$, is defined by $\mathrm{rk}_{\mathcal F}(\emptyset)$.
\end{definition}

    The rank of a front is well-defined by the previous proposition.
    We also note the following lemma.
    Its proof is straightforward, and we refer to \cite{todorcevic2010introduction} for details.
\begin{lemma}
    Let $M \in [\omega]^{\omega}$ and $\mathcal F \subseteq [M]^{<\omega}$ be a front on $M$, and let $s \in \mathcal T_{\mathcal F}$.
    Then:
    \begin{enumerate}[label=(\alph*)]
        \item $\mathcal F_s$ is a front on $M\restriction s$.
        \item If, additionally, $\mathcal F$ is a barrier, then $\mathcal F_s$ is a barrier on $M\restriction_s$.
        \item $\mathrm{rk}_{\mathcal F}(s\cup t)=\mathrm{rk}_{\mathcal F_s}(t)$ for every $t \in \mathcal T_{\mathcal F_s}$.
        \item In particular, $\mathrm{rk}(\mathcal F_s)=\mathrm{rk}_{\mathcal F}(s)$ for every $s \in \mathcal T_{\mathcal F}$, and if $s\neq\emptyset$, then $\mathrm{rk}(\mathcal F_s)<\mathrm{rk}(\mathcal F)$.
    \end{enumerate}
\end{lemma}

The rank of a front or barrier need not be preserved under restrictions. For
recursive arguments, we therefore use uniform fronts and barriers, whose rank
is preserved under restriction.
\begin{definition}
    Let $M\subseteq \omega$ be infinite and $\mathcal{F}\subseteq [M]^{<\omega}$ be a front of rank $\alpha$.
    We define by induction on $\alpha$ that $\mathcal F$ is a \emph{uniform front on $M$} if one of the following holds:
    \begin{itemize}
        \item $\alpha=0$ and $\mathcal F=\{\emptyset\}$, or
        \item $\alpha=\beta+1$ for some $\beta<\omega_1$ and, for every $n \in M$, $\mathcal{F}_{\{n\}}$ is a uniform front on $M\setminus n+1$ of rank $\beta$, or
        \item $\alpha$ is a limit ordinal and, for every $n \in M$, $\mathcal{F}_{\{n\}}$ is a uniform front on $M\setminus n+1$ of some rank $\alpha_n<\alpha$, $(\alpha_n)_{n \in M}$ is strictly increasing and $\sup_{n \in M} \alpha_n=\alpha$.
    \end{itemize}

    A \emph{uniform barrier on $M$} is a uniform front on $M$ which is also a barrier.
\end{definition}

Typical examples of uniform barriers are the families $[\omega]^n$, for $n\in\omega$, and the Schreier barrier.
Inductively, one easily sees that for every $n\in\omega$, $[\omega]^n$ is a uniform barrier on $\omega$.
The \emph{Schreier barrier} is defined by
\[
    \{s\subseteq\omega : s\neq\emptyset,\ \min s+1=|s|\}.
\]
It is easily seen to be a uniform barrier of rank $\omega$.

Restrictions of uniform fronts and barriers preserve uniformity and rank
\cite{todorcevic2010introduction}.
\begin{lemma}
    Let $M\subseteq \omega$ be infinite and $\mathcal{F}\subseteq [M]^{<\omega}$ be a uniform front (barrier) on $M$ of rank $\alpha$.
    Then, for every $N \in [M]^{\omega}$, $\mathcal{F}\restriction N$ is a uniform front (barrier) on $N$ of rank $\alpha$.
\end{lemma}

\subsection{Topological notions}
Finite-dimensional Ramsey versions of sequential compactness were introduced
in \cite{kubis2023topological} and extended to arbitrary barriers in
\cite{corral2024infinite}, while barrier versions of countable compactness were
introduced in \cite{corral2024highCompactness}.
We recall $\mathcal B$-countable compactness and then introduce its stronger recursive
variant, which we call \emph{cascade countable compactness}.

To define these notions, we first associate a recursively defined filter
$\mathcal F^{\mathcal B}$ with each pair $(\mathcal F,\mathcal B)$.

\begin{definition}
    Let $M\subseteq \omega$, $\mathcal F$ be a free filter on $\omega$ with $M \in \mathcal F$ and $\mathcal B$ be a front on $M$.
    
    We define $\mathcal F^{\mathcal B}$ as follows:

    \begin{enumerate}[label=(\alph*)]
        \item If $\mathrm{rk}(\mathcal B)=0$, that is, if $\emptyset\in \mathcal B$ then $\mathcal F^{\mathcal B}=\{\{\emptyset\}\}$.
        \item Otherwise, we (recursively) define:
        \begin{equation*}
            \mathcal F^{\mathcal B}=\left\{A\subseteq \mathcal B:\left\{n \in M: \left\{s \in \mathcal B_{\{n\}}: \{n\}\cup s \in A\right\}\in \mathcal F^{\mathcal B_{\{n\}}}\right\}\in \mathcal F\right\}.
        \end{equation*}
    \end{enumerate}
\end{definition}

This recursive definition is well-founded, since for every $n \in M$, $M\setminus n+1\in\mathcal F$, $\mathcal B_{\{n\}}$ is a front on $M\setminus n+1$, and $\mathrm{rk}(\mathcal B_{\{n\}})<\mathrm{rk}(\mathcal B)$.

It is straightforward to verify that $\mathcal F^{\mathcal B}$ is a filter on $\mathcal B$ and that
\[
\mathcal F^{[\omega]^1}
=
\bigl\{\{\{n\}: n\in A\}: A\in\mathcal F\bigr\}.
\]
Under the identification of $n\in\omega$ with $\{n\}\in[\omega]^1$,
$\mathcal F^{[\omega]^1}$ is identified with $\mathcal F$.

The filter $\mathcal F^{\mathcal B}$ is obtained by iterating $\mathcal F$
along the tree of the front $\mathcal B$.

Recall that if $Z$ is a set, $X$ is a topological space, $\mathcal F$ is a filter with $Z\in\mathcal F$, and $f$ is a function into $X$ whose domain contains $Z$, then $x\in X$ is an $\mathcal F$-limit point of $f$ if, for every neighborhood $U$ of $x$, one has
$\{z\in Z: f(z)\in U\}\in\mathcal F$.

\begin{definition}
    Let $X$ be a topological space, $M\subseteq \omega$ be infinite, and $\mathcal B\subseteq [M]^{<\omega}$ be a front on $M$.
    We say that $X$ is \emph{$\mathcal B$-countably compact} if for every $f:\mathcal B\to X$, there exists a free ultrafilter $p\in\beta\omega$ with $M\in p$ such that $f$ has a $p^{\mathcal B}$-limit point in $X$.
    If $\mathcal B=[\omega]^n$, we say that $X$ is \emph{$n$-countably compact}.
\end{definition}

The previous definition only asks for a limit point with respect to the induced filter on $\mathcal B$.
The stronger notion below requires the limiting process to be witnessed recursively along the derived fronts $\mathcal B_{\{n\}}$.

\begin{definition}\label{def:cascadeLimitPoint}
    Let $X$ be a topological space, let $M\subseteq\omega$ be infinite, let
    $\mathcal B\subseteq[M]^{<\omega}$ be a front on $M$, let $\mathcal F$ be a
    free filter on $\omega$ with $M\in\mathcal F$, and let $x\in X$.

    Given a function $f$ into $X$ whose domain contains $\mathcal B$, if $\mathrm{rk}(\mathcal B)>0$ and $n \in M$, we define $f_n:\mathcal B_{\{n\}}\to X$ as $f_n(s)=f(\{n\}\cup s)$.

    We say that $x$ is a \emph{$\mathcal B$-cascade $\mathcal F$-limit point} of $f$ if one of the following holds.
    \begin{enumerate}[label=(\alph*)]
        \item If $\mathrm{rk}(\mathcal B)=0$, that is, if $\emptyset\in \mathcal B$, then $x$ is a $\mathcal B$-cascade $\mathcal F$-limit point of $f$ if and only if $f(\emptyset)=x$.
        \item Otherwise, if $\mathrm{rk}(\mathcal B)>0$, then $x$ is a $\mathcal B$-cascade $\mathcal F$-limit point of $f$ if there exists $A \in \mathcal F$ with $A\subseteq M$ and a sequence $(x_n: n \in A)$ such that, for every $n \in A$, $x_n$ is a $\mathcal B_{\{n\}}$-cascade $\mathcal F$-limit point of $f_n$ and $x$ is an $\mathcal F$-limit point of $(x_n: n \in A)$.
        Here, $\mathcal B_{\{n\}}$ is seen as a front on $M\restriction_{>n}$.
    \end{enumerate} 
\end{definition}

\begin{definition}
    Let $X$ be a topological space, $M\subseteq \omega$ be infinite, and $\mathcal B\subseteq [M]^{<\omega}$ be a front on $M$.
    We say that $X$ is \emph{$\mathcal B$-cascade countably compact} if for every $f:\mathcal B\to X$, there exists a free ultrafilter $p\in\beta\omega$ with $M\in p$ such that $f$ has a $\mathcal B$-cascade $p$-limit point in $X$.

        If $\mathcal B=[\omega]^n$, we say that $X$ is \emph{$n$-cascade countably compact}.
\end{definition}

Under the standard identification of $[\omega]^2$ with increasing pairs in
$\omega^2$, $[\omega]^2$-cascade countable compactness agrees with the notion
of being \emph{doubly countably compact} defined in \cite{banakh2009rees}.
This terminology also appears in \cite{corral2024highCompactness}.
Thus, the definition above generalizes the notion of being doubly countably compact to arbitrary fronts.
\begin{lemma}
    Let $X$ be a topological space, $M\subseteq \omega$ be infinite, $\mathcal B\subseteq [M]^{<\omega}$ be a front on $M$, $f:\mathcal B\to X$, $x \in X$ and $\mathcal F$ be a free filter on $\omega$ with $M \in \mathcal F$.

    If $x$ is a $\mathcal B$-cascade $\mathcal F$-limit point of $f$, then $x$ is an $\mathcal F^{\mathcal B}$-limit point of $f$.

    Thus, every $\mathcal B$-cascade countably compact space is $\mathcal B$-countably compact.
\end{lemma}
\begin{proof}
We proceed by induction on the rank of $\mathcal B$.

If $\mathrm{rk}(\mathcal B)=0$, then $\emptyset \in \mathcal B$ and $f(\emptyset)=x$, so $x$ is an $\mathcal F^{\mathcal B}$-limit point of $f$.

Now, assume that $\mathrm{rk}(\mathcal B)>0$ and that the result holds for every front of rank $<\mathrm{rk}(\mathcal B)$.
Let $U$ be a neighborhood of $x$.
Let $A \in \mathcal F$ with $A\subseteq M$ and $(x_n: n \in A)$ be such that for every $n \in A$, $x_n$ is a $\mathcal B_{\{n\}}$-cascade $\mathcal F$-limit point of $f_n$ and $x$ is an $\mathcal F$-limit point of $(x_n: n \in A)$.

Now let $A_U=\{n\in A:x_n\in U\}$.
Since $x$ is an $\mathcal F$-limit point of $(x_n:n\in A)$, we have $A_U\in\mathcal F$.

For every $n\in A_U$, by the inductive hypothesis,
\[
f_n^{-1}[U]\in \mathcal F^{\mathcal B_{\{n\}}}.
\]

Observe that $f_n^{-1}[U]=\{s \in \mathcal B_{\{n\}}: f(\{n\}\cup s) \in U\}=\{s \in \mathcal B_{\{n\}}: \{n\}\cup s \in f^{-1}[U]\}$.

We must conclude that $f^{-1}[U]\in \mathcal F^{\mathcal B}$, that is, that $\{n \in M: \{s \in \mathcal B_{\{n\}}: \{n\}\cup s \in f^{-1}[U]\}\in \mathcal F^{\mathcal B_{\{n\}}}\}\in \mathcal F$.
However, we have just seen that $A_U$ is contained in the former set.
Since $A_U\in\mathcal F$ and $\mathcal F$ is upward closed, the former set belongs to $\mathcal F$.
\end{proof}

The converse fails in general, as our later examples will show.

The same inductive argument, now keeping track of restrictions, yields the following.

\begin{lemma}
    Let $X$ be a topological space, $M\subseteq \omega$ be infinite, $N \in [M]^{\omega}$, $\mathcal B\subseteq [M]^{<\omega}$ be a front on $M$, $f:\mathcal B\to X$, $x \in X$, and $\mathcal F$ be a free filter on $\omega$ with $N \in \mathcal F$.

    If $x$ is a $(\mathcal B\restriction N)$-cascade $\mathcal F$-limit point of $f$, then $x$ is an $\mathcal F^{\mathcal B}$-limit point of $f$.

    In particular, every $(\mathcal B\restriction N)$-cascade countably compact space is $\mathcal B$-countably compact.
\end{lemma}

As ultrafilter limits are unique in Hausdorff spaces, the following is an immediate consequence of the previous lemma.
\begin{corollary}\label{corollary:uniqueCascadeLimit}
    Let $X$ be a Hausdorff topological space, $M\subseteq \omega$ be infinite, $\mathcal B\subseteq [M]^{<\omega}$ be a front on $M$, $f:\mathcal B\to X$, and $p$ be a free ultrafilter on $\omega$ with $M \in p$.

    Then $f$ has at most one $\mathcal B$-cascade $p$-limit point in $X$.
\end{corollary}

    \section{On discrete images of barriers}\label{section:discrete-images}
It is well known that every infinite Hausdorff space has an infinite discrete subset.
In \cite{carlson2005discrete}, Carlson, Hindman, and Strauss proved the following strengthening for functions defined on $[N]^n$, where $N$ is an infinite countable set and $n \in \omega$.

\begin{proposition}[\cite{carlson2005discrete}]\label{prop:discreteImageCarlson}
    Let $X$ be a Hausdorff space, let $n \in \omega$, let $N$ be infinite and countable, and let $f: [N]^n \to X$.
    Then there exists $A \in [N]^{\omega}$ such that $f[[A]^n]$ is a discrete subset of $X$.
\end{proposition}

Since for every barrier $\mathcal B$ of finite rank there exist $M \in [\omega]^{\omega}$ and $n \in \omega$ such that $\mathcal B\restriction M = [M]^n$, Proposition~\ref{prop:discreteImageCarlson} yields the analogous conclusion for maps from finite-rank barriers into Hausdorff spaces.
The next propositions show that the same conclusion fails for barriers of
infinite rank.

\begin{proposition}\label{prop:discreteImageSchreier}
Let $\mathcal B$ be the Schreier barrier.
Then there exist a countable linearly ordered topological space $L$, namely $\omega^{<\omega}$ endowed with the Kleene--Brouwer order, and a function $f:\mathcal B\to L$ such that for no infinite $A\subseteq\omega$, the set $f[\mathcal B\restriction A]$ is discrete.
\end{proposition}
\begin{proof}
    If $u, v \in \omega^{<\omega}$ are such that $u\not\subseteq v$ and $v\not \subseteq u$, let $k_{uv}=\min\{k: u(k)\neq v(k)\}$.

    Consider $L=\omega^{<\omega}$ ordered as follows:

    \[
        u\leq v \quad\text{iff}\quad v\subseteq u,\ \text{or else }u,v\text{ are }\subseteq\text{-incomparable and }u(k_{uv})<v(k_{uv}).
    \]

    $L$ is known as the Kleene--Brouwer order on $\omega^{<\omega}$ (also known as the Lusin--Sierpi\'nski order).

    \begin{claim}
        For every $u\in L$ and every open set $W$ containing $u$, one has $u^\frown t\in W$ for all but finitely many $t\in\omega$.
    \end{claim}

    \begin{proof}[Proof of Claim]
        Since $L$ has no minimal elements, there exists $x<u$ such that $(x,u]\subseteq W$.
        Then either $u\subsetneq x$ or $x, u$ are $\subseteq$-incomparable and $x(k_{xu})<u(k_{xu})$.
        In the first case, for every $t>x(|u|)$, $u^\frown t \in (x, u]$.
        In the second case, for every $t\in \omega$, $u^\frown t \in (x, u]$.
    \end{proof}

    \begin{claim}
        For every infinite $X\subseteq \omega$, every $d\in\omega$, and every $1\leq n<m$ in $\omega$, the set $\{u \in X^n: d<u(0)<\dots<u(n-1)\}$ is contained in the closure of $\{u \in X^m: d<u(0)<\dots<u(m-1)\}$.
    \end{claim}

    \begin{proof}[Proof of Claim]
        It suffices to prove the case $m=n+1$, so the general case follows by induction on $m$.
        Let $d\in\omega$ and $u\in X^n$ be such that $d<u(0)<\dots<u(n-1)$.
        By the previous claim, $u$ belongs to the closure of
        \[
            Y=\{u^\frown t: t\in X\}.
        \]
        Since only finitely many $t\in X$ satisfy $t\leq u(n-1)$, all but finitely many elements of $Y$ belong to
        \[
            \{v \in X^{n+1}: d<v(0)<\dots<v(n)\}.
        \]
        Therefore, $u\in \overline{\{v \in X^{n+1}: d<v(0)<\dots<v(n)\}}$.
        This proves the claim.
    \end{proof}

    Now let $\mathcal B$ be the Schreier barrier.
    \[
        \mathcal B=\bigcup_{n \in \omega}\left\{\{n\}\cup s: s \in [\omega\setminus (n+1)]^n\right\}.
    \]

    Define $f:\mathcal B\to L$ by $f(\{n\}\cup s)=\bar s$, where $\bar s$ is $s$ written as an increasing sequence.
    We claim that for no infinite $A$, $f[\mathcal B\restriction A]$ is discrete.
    To see this, let $A\subseteq \omega$ be infinite, let $1\leq a<b$ be the second and third elements of $A$, and let $X=A\setminus\{a,b\}$. Then
    \[
        S_b=\{u \in X^b: b<u(0)<\dots<u(b-1)\}
    \]
    is contained in $f[\mathcal B\restriction A]$, since for every
    $u=(u(0), \dots, u(b-1))\in S_b$ we have
    \[
        \{b\}\cup \{u(0), \dots, u(b-1)\}\in \mathcal B\restriction A
    \]
    and
    \[
        f(\{b\}\cup \{u(0), \dots, u(b-1)\})=u.
    \]
    Choose $u\in S_b$, and let $w=u\restriction a=(u(0), \dots, u(a-1))$.
    Then $w\in f[\mathcal B\restriction A]$, because
    $\{a\}\cup \{u(0), \dots, u(a-1)\}\in \mathcal B\restriction A$
    and $f(\{a\}\cup \{u(0), \dots, u(a-1)\})=w$.

    Moreover, by the previous claim, $w\in \overline{S_b}$.
    Since $w\notin S_b$, it follows that $w$ is an accumulation point of $f[\mathcal B\restriction A]$.
    Therefore, $f[\mathcal B\restriction A]$ is not discrete.
\end{proof}

The same failure occurs for every uniform barrier of infinite rank. We use the
comparison relation between barriers introduced in \cite{corral2024infinite}.

\begin{definition}[{\cite[Definition 2.19]{corral2024infinite}}]
    Given two barriers $\mathcal B$ and $\mathcal C$ on a set $M\in [\omega]^\omega$, we write $\mathcal B \preceq \mathcal C$ if there is a finite-to-one, non-decreasing function $h \in \omega^\omega$ such that for every infinite subset $M' \in [M]^\omega$ there exists $N \in [M']^\omega$ such that $h\restriction N$ is one-to-one and
    \[
        \forall c \in (\mathcal C\restriction N)\ \exists b \in \mathcal B \, (b \sqsubseteq h[c]).
    \]
\end{definition}

\begin{proposition}[{\cite[Proposition 2.20]{corral2024infinite}}]\label{prop:barrierPreceq}
    Let $\mathcal B, \mathcal C$ be two barriers on $\omega$ such that
    $\mathrm{rk}(\mathcal B)\leq \mathrm{rk}(\mathcal C)$ and $\mathcal C$ is uniform.
    Then $\mathcal B \preceq \mathcal C$.
\end{proposition}

\begin{proposition}
    Let $\mathcal C$ be a uniform barrier on $\omega$ of infinite rank.
    Then there exists a function $g: \mathcal C \to L$, where $L$ is the Kleene--Brouwer order on $\omega^{<\omega}$, such that for no infinite $A\subseteq\omega$, the set $g[\mathcal C\restriction A]$ is discrete.
\end{proposition}

\begin{proof}
    Let $\mathcal B$ be the Schreier barrier and let $f:\mathcal B\to L$ be as in Proposition~\ref{prop:discreteImageSchreier}.
    Since $\mathcal C$ is uniform and $\mathrm{rk}(\mathcal B)=\omega\leq \mathrm{rk}(\mathcal C)$, by Proposition~\ref{prop:barrierPreceq} we have $\mathcal B\preceq\mathcal C$.
    Hence, there exists a finite-to-one, non-decreasing function $h\in\omega^\omega$ such that for every infinite $M\in[\omega]^\omega$ there exists $N\in[M]^\omega$ such that $h\restriction N$ is one-to-one and
    \[
        \forall c\in (\mathcal C\restriction N)\ \exists b\in\mathcal B\ (b\sqsubseteq h[c]).
    \]

    Fix $b_0\in\mathcal B$.
    Define $\phi:\mathcal C\to\mathcal B$ as follows: if there exists $b\in\mathcal B$ such that $b\sqsubseteq h[c]$, let $\phi(c)=b$; otherwise, let $\phi(c)=b_0$.

    Notice that there is at most one such $b$, because if $b_1,b_2\in\mathcal B$ satisfy $b_1\sqsubseteq h[c]$ and $b_2\sqsubseteq h[c]$, then $b_1$ and $b_2$ are $\sqsubseteq$-comparable, and since $\mathcal B$ is Nash-Williams, it follows that $b_1=b_2$.

    Let $g=f\circ\phi:\mathcal C\to L$.
    We claim that for no infinite $M\subseteq\omega$, the set $g[\mathcal C\restriction M]$ is discrete.

    To see this, let $M\subseteq\omega$ be infinite, and choose $N\in[M]^\omega$ such that $h\restriction N$ is one-to-one and
    \[
        \forall c\in (\mathcal C\restriction N)\ \exists b\in\mathcal B\ (b\sqsubseteq h[c]).
    \]
    Since $h\restriction N$ is one-to-one and non-decreasing, it is strictly increasing, so $h[N]$ is infinite.

    We first show that
    \[
        \phi[\mathcal C\restriction N]=\mathcal B\restriction h[N].
    \]

    To see that $\phi[\mathcal C\restriction N]\subseteq \mathcal B\restriction h[N]$, note that if $c\in\mathcal C\restriction N$, then $\phi(c)\sqsubseteq h[c]$, and since $h[c]\subseteq h[N]$, we have $\phi(c)\in\mathcal B\restriction h[N]$.

    Conversely, let
    \[
        \mathcal F=\phi[\mathcal C\restriction N].
    \]
    We claim that $\mathcal F$ is a front on $h[N]$.
    Let $A\in[h[N]]^\omega$, and set
    \[
        N_A=h^{-1}[A]\cap N.
    \]
    Then $N_A\in[N]^\omega$.
    Since $\mathcal C$ is a front on $\omega$, there exists $c\in\mathcal C\restriction N_A$ such that $c\sqsubseteq N_A$.
    As $c\in\mathcal C\restriction N$, we have $\phi(c)\sqsubseteq h[c]$.
    Moreover, because $h\restriction N$ is strictly increasing and $c\sqsubseteq N_A$, we get
    \[
        h[c]\sqsubseteq A.
    \]
    Hence,
    \[
        \phi(c)\sqsubseteq A,
    \]
    proving that $\mathcal F$ is a front on $h[N]$.

    Since $\mathcal F\subseteq \mathcal B\restriction h[N]$ and $\mathcal B\restriction h[N]$ is a barrier on $h[N]$, it follows that $\mathcal F=\mathcal B\restriction h[N]$.
    Indeed, let $s\in\mathcal B\restriction h[N]$.
    Extend $s$ to some infinite $A\in[h[N]]^\omega$ with $s\sqsubseteq A$.
    Since $\mathcal F$ is a front on $h[N]$, there exists $t\in\mathcal F$ such that $t\sqsubseteq A$.
    As $t\in\mathcal F\subseteq \mathcal B\restriction h[N]$, both $s$ and $t$ belong to the barrier $\mathcal B\restriction h[N]$ and are initial segments of the same infinite set $A$.
    Therefore, $s=t$.
    Thus, $s\in\mathcal F$, and so $\mathcal B\restriction h[N]\subseteq \mathcal F$.

    Hence,
    \[
        \phi[\mathcal C\restriction N]=\mathcal B\restriction h[N].
    \]
    Therefore,
    \[
        g[\mathcal C\restriction N]
        =
        f[\phi[\mathcal C\restriction N]]
        =
        f[\mathcal B\restriction h[N]].
    \]
    By Proposition~\ref{prop:discreteImageSchreier}, the set $f[\mathcal B\restriction h[N]]$ is not discrete.
    Hence, $g[\mathcal C\restriction N]$ is not discrete.

    Since $N\subseteq M$, we have $g[\mathcal C\restriction N]\subseteq g[\mathcal C\restriction M]$, so $g[\mathcal C\restriction M]$ is not discrete either.
    As $M$ was arbitrary, this proves the result.
\end{proof}

The discrete-thinning result from \cite{carlson2005discrete} is used in
Section~\ref{section:betaomega} to treat finite cascade levels. Since it fails
for uniform barriers of infinite rank, the group constructions in
Section~\ref{section:boolean-groups} use a different argument.

    \section{Nonproductivity in subspaces of \texorpdfstring{$\beta\omega$}{beta omega}}\label{section:betaomega}

We identify points of $\beta\omega$ with ultrafilters on $\omega$.

A selective ultrafilter, also called a Ramsey ultrafilter, is a free ultrafilter $p$ on $\omega$ such that for every map $f:[\omega]^2\to\{0,1\}$, there exists $A\in p$ such that $f\restriction [A]^2$ is constant.
The existence of selective ultrafilters is independent of ZFC: for instance, many of them exist under the Continuum Hypothesis or Martin's Axiom (see, e.g., \cite[Theorem~2]{blass1973rudin}).
However, they do not exist after adding $\aleph_2$ random reals to a model satisfying the generalized continuum hypothesis \cite{kunen1976somePoints}.

Assuming the existence of selective ultrafilters, the following results were obtained in \cite{corral2024highCompactness}.

\begin{lemma*}[{\cite[Lemma~4.5]{corral2024highCompactness}}]
Assume the existence of $\kappa$ many selective ultrafilters. If $A \subseteq \beta\omega$ and
 $|A| < \kappa$, then $\beta\omega \setminus A$ is $2$-cascade countably compact.
\end{lemma*}

\begin{theorem*}[{\cite[Theorem~4.8]{corral2024highCompactness}}]
    Assume the existence of a selective ultrafilter. Then there is a $2$-cascade countably compact space whose square is not countably compact.
\end{theorem*}

\begin{theorem*}[{\cite[Theorem~4.10]{corral2024highCompactness}}]
    Assume the existence of a selective ultrafilter.
    Then there exists a $2$-countably compact space which is not $2$-cascade countably compact.
\end{theorem*}

Using Proposition~\ref{prop:discreteImageCarlson}, we obtain ZFC versions of
the preceding results for all finite dimensions. We first prove that
$\beta\omega\setminus A$ is $n$-cascade countably compact for every $n<\omega$
whenever $A\subseteq\beta\omega$ and $|A|<2^{\mathfrak c}$. We then construct
a Tychonoff space that is $n$-cascade countably compact for every $n<\omega$
but whose square is not countably compact, and a space that is $k$-countably
compact for every $k<\omega$ but is not $2$-cascade countably compact.

To pass from dimension $2$ to higher finite dimensions, we use the
Erd\H{o}s--Rado canonical theorem.
For a modern proof, see \cite{todorcevic2010introduction}.

\begin{theorem}[Erd\H{o}s--Rado canonical theorem \cite{erdos1950combinatorial}]
    Let $k \in \omega$ and let $\varphi : [\omega]^k \to [\omega]^{<\omega}$.
    Then there exist an infinite set $M \subseteq \omega$ and a set $I \subseteq \{0, \ldots, k-1\}$ such that, for all
    $s = \{n_0 < \cdots < n_{k-1}\}, t = \{m_0 < \cdots < m_{k-1}\} \in [M]^k$,
    we have
    \[
        \varphi(s) = \varphi(t)
        \quad\text{if and only if}\quad
        n_i = m_i \text{ for all } i \in I.
    \]
\end{theorem}

The following lemma provides the induction step from dimension $d$ to dimension
$d+1$.
\begin{lemma}\label{lemma:betaOmegaDimensionRaising}
    Let $d\in\omega$, let $A \in [\omega]^{\omega}$, let $X\subseteq \beta\omega$, and let
    $f:[A]^d\to \beta\omega\setminus X$.
    Then there exists $B \in [A]^{\omega}$ such that
    \[
        \left|\left\{p \in B^*: f \text{ has no } d\text{-cascade } p\text{-limit point in } \beta\omega\setminus X
        \right\}\right|
        \leq \max\{|X|,\aleph_0\}.
    \]
\end{lemma}
\begin{proof}
    We proceed by induction on $d$.

    If $d=0$, the result is trivial.
    For $d=1$ the claim is easy since any $f:A\rightarrow \beta\omega$ has a restriction which is either constant or one-to-one with discrete range.

    Assume the claim holds for $d\geq 1$ and let
    \[
    f:[A]^{d+1}\to \beta\omega\setminus X.
    \]

    By Proposition~\ref{prop:discreteImageCarlson}, we may assume that $f[[A]^{d+1}]$ is a discrete subset of $\beta\omega\setminus X$.

    Recursively, define a decreasing sequence $(A_k: k \in \omega)$ and a sequence $(a_k: k \in \omega)$ such that:
    \begin{enumerate}[label=(\alph*)]
        \item $A_0=A$.
        \item $a_k=\min A_k$ for every $k \in \omega$.
        \item $A_{k+1}\in [A_k\setminus \{a_k\}]^{\omega}$ is such that
        \[
        \left|\left\{p \in A_{k+1}^*: f_{a_k} \text{ has no } d\text{-cascade } p\text{-limit point in } \beta\omega\setminus X\right\}\right|
        \leq \max\{|X|,\aleph_0\},
        \]
        which exists by the inductive hypothesis.
    \end{enumerate}

    Let $C=\{a_k: k \in \omega\}$.
    Fix $k \in \omega$ and let $c=a_k$.
    Since $C\setminus (c+1)\subseteq A_{k+1}$, for every $p \in C^*$ we have $A_{k+1}\in p$.
    Hence, by the choice of $A_{k+1}$,
    \[
    \left|\left\{p \in C^*: f_c \text{ has no } d\text{-cascade } p\text{-limit point in } \beta\omega\setminus X\right\}\right|
    \leq \max\{|X|,\aleph_0\}.
    \]

    By the Erd\H{o}s--Rado theorem, there exist $B \in [C]^{\omega}$ and $I\subseteq \{0,\ldots,d\}$ such that for every two distinct
    \[
    \{a_0,\ldots,a_d\},\ \{b_0,\ldots,b_d\}\in [B]^{d+1},
    \]
    with $a_0<\cdots<a_d$ and $b_0<\cdots<b_d$,
    we have $f(\{a_0, \ldots, a_d\})=f(\{b_0, \ldots, b_d\})$ if and only if $a_i=b_i$ for every $i \in I$.

    \textbf{Case 1: $0\in I$.}
    In this case, for every two distinct $k,l\in B$,
    \[
        f_k\big[[B\setminus (k+1)]^d\big]\cap f_l\big[[B\setminus (l+1)]^d\big]=\emptyset.
    \]
    For each $k\in B$, let
    \[
        E_k=\left\{p\in B^*: f_k \text{ has no } d\text{-cascade } p\text{-limit point in } \beta\omega\setminus X\right\}.
    \]
    Since $B^*\subseteq C^*$, we have $|E_k|\leq \max\{|X|,\aleph_0\}$ for every $k\in B$.
    Let $E=\bigcup_{k\in B} E_k$.
    Then $|E|\leq \max\{|X|,\aleph_0\}$.
    For each $p\in B^*\setminus E$ and each $k\in B$, let $x_{p,k}\in \beta\omega\setminus X$ be a $d$-cascade $p$-limit point of $f_k$.
    
    Now fix distinct $p,q\in B^*\setminus E$.
    Let $D\in p$ and $F\in q$ be disjoint.
    Put
    \[
    D_f=\bigcup_{k\in D} f_k\big[[B\setminus (k+1)]^d\big]
    \quad\text{and}\quad
    F_f=\bigcup_{k\in F} f_k\big[[B\setminus (k+1)]^d\big].
    \]
    
    Since $D_f$ and $F_f$ are disjoint subsets of the discrete set $f[[B]^{d+1}]$ in $\beta\omega$, we have $\cl D_f\cap \cl F_f=\emptyset$.
    Moreover, $x_{p, k}\in \cl D_f$ for every $k \in D$ and $x_{q, k}\in \cl F_f$ for every $k \in F$.
    Since $\beta\omega$ is compact, the limits $p\text{-}\lim_{k \in B} x_{p, k}$ and $q\text{-}\lim_{k \in B} x_{q, k}$ exist and belong to $\cl D_f$ and $\cl F_f$, respectively.
    Thus, $p\text{-}\lim_{k \in B} x_{p, k}\in \cl D_f$ and $q\text{-}\lim_{k \in B} x_{q, k}\in \cl F_f$ cannot be equal.
    Let $x_p=p\text{-}\lim_{k \in B} x_{p, k}$ for each $p\in B^*\setminus E$.

    Thus, the map $p\mapsto x_p$ is injective on $B^*\setminus E$.
    Let $F_X=\{p\in B^*\setminus E: x_p\in X\}$.
    Then $|F_X|\leq \max\{|X|,\aleph_0\}$.
    Hence, $|E\cup F_X|\leq \max\{|X|,\aleph_0\}$.

    Finally, if $p\in B^*\setminus (E\cup F_X)$, then for every $k\in B$, $x_{p,k}$ is a $d$-cascade $p$-limit point of $f_k$, and $x_p\in \beta\omega\setminus X$ is the $p$-limit of the sequence $(x_{p,k}:k\in B)$.
    Therefore, $x_p$ is a $(d+1)$-cascade $p$-limit point of $f$ in $\beta\omega\setminus X$.

    \textbf{Case 2: $0\notin I.$}
    Let $l=\min B$.
    For every $k\in B\setminus\{l\}$ and every $s\in [B\setminus (k+1)]^d$, the sets $\{l\}\cup s \quad\text{and}\quad \{k\}\cup s$ agree on all coordinates indexed by $I$.
    Since $0\notin I$, it follows that $f(\{l\}\cup s)=f(\{k\}\cup s)$.
    Hence, $f_k=f_l\restriction [B\setminus (k+1)]^d$
    for every $k\in B\setminus\{l\}$.

    By the inductive hypothesis, choose
    $Z\in[B\setminus(l+1)]^\omega$ and put
    \[
        T=\left\{p\in Z^*: f_l \text{ has no } d\text{-cascade } p\text{-limit point in } \beta\omega\setminus X\right\}.
    \]
    We may choose $Z$ so that $|T|\leq\max\{|X|,\aleph_0\}$.

    Fix $p\in Z^*\setminus T$, and let $x_p\in \beta\omega\setminus X$ be a $d$-cascade $p$-limit point of $f_l$.
    For every $k\in Z$, since $Z\setminus (k+1)\in p$ and
    \[
    f_k=f_l\restriction [Z\setminus (k+1)]^d,
    \]
    the same point $x_p$ is a $d$-cascade $p$-limit point of $f_k$.
    Therefore, taking the constant sequence $(x_p:k\in Z)$, we conclude that $x_p$ is a $(d+1)$-cascade $p$-limit point of $f$ in $\beta\omega\setminus X$.
\end{proof}

Now we are ready to prove the main result of this section: an improved version of \cite[Lemma~4.5]{corral2024highCompactness} which does not require the existence of selective ultrafilters.

\begin{proposition}\label{proposition:betaOmegaSmallDimensional}
    For every $n \in \omega$ and every $X\subseteq \beta\omega$ such that $|X|<2^{\mathfrak c}$, $\beta\omega\setminus X$ is $n$-cascade countably compact.
\end{proposition}
\begin{proof}
    Let $f:[\omega]^n\to \beta\omega\setminus X$.
    By Lemma~\ref{lemma:betaOmegaDimensionRaising}, there exists $B\in [\omega]^{\omega}$ such that
    \[
        \left|\left\{p \in B^*: f \text{ has no } n\text{-cascade } p\text{-limit point in } \beta\omega\setminus X\right\}\right|
        \leq \max\{|X|,\aleph_0\}.
    \]

    Since $|B^*|=2^{\mathfrak c}>\max\{|X|,\aleph_0\}$, choose $p\in B^*$
    outside the set above. Then $f$ has an $n$-cascade
    $p$-limit point in $\beta\omega\setminus X$.
\end{proof}

Using Proposition~\ref{proposition:betaOmegaSmallDimensional} and arguing as in \cite{corral2024highCompactness}, we obtain the following.

\begin{theorem}\label{theorem:betaOmegaHighDimensionalProduct}
    There exist subspaces $X$ and $Y$ of $\beta\omega$, both containing $\omega$, such that
    $X$ and $Y$ are $n$-cascade countably compact for every $n \in \omega$, while
    $X\times Y$ is not countably compact.

    In fact, the sequence $((n,n):n\in\omega)$ has no accumulation points in $X\times Y$.
\end{theorem}
\begin{proof}
    Let $\theta$ be an uncountable regular cardinal such that $\beta\omega \in H(\theta)$.
    Let $M$ be a $\aleph_0$-closed elementary submodel of $H(\theta)$ of cardinality $\mathfrak c$ with $\beta\omega\in M$ (which exists by, e.g., \cite[Lemma~III.8.4]{kunen2011set}).
    
    Define $X=\beta\omega\cap M$.
    We claim that $X$ is $n$-cascade countably compact for every $n\in\omega$.
    Indeed, fix $n\in\omega$ and let $f:[\omega]^n\to X$.
    Since $\ran(f)\subseteq X\subseteq M$ and $M$ is $\aleph_0$-closed, we have $f\in M$.
    By elementarity, there exist $p\in\omega^*$ and $x\in\beta\omega$ such that $x$ is an $n$-cascade $p$-limit point of $f$, and such witnesses may be taken in $M$.
    Hence, $x\in\beta\omega\cap M=X$, proving the claim.

    Since $|X\setminus \omega|\leq |M|=\mathfrak c<2^{\mathfrak c}$, by Proposition~\ref{proposition:betaOmegaSmallDimensional}, $Y=\beta\omega\setminus (X\setminus \omega)$ is $n$-cascade countably compact for every $n \in \omega$.

    The sequence $((n,n):n\in\omega)$ has no accumulation points in $X\times Y$.
    Indeed, any accumulation point of this sequence in $\beta\omega\times\beta\omega$ must be of the form $(p,p)$ with $p\in\omega^*$.
    If such a point belonged to $X\times Y$, then its first coordinate would belong to $X\setminus\omega$ and its second coordinate to $Y\setminus\omega$, which are disjoint.
\end{proof}

Thus, two spaces may be $n$-cascade countably compact for every $n<\omega$
while their product is not countably compact.

By iterating the previous construction, we obtain the following.

\begin{proposition}\label{proposition:betaOmegaManyHighDimensional}
    There exists a family $(C_\alpha: \alpha<2^{\mathfrak c})$ of subspaces of $\beta\omega$ which are $n$-cascade countably compact for every $n \in \omega$ such that for every two distinct $\alpha, \beta<2^{\mathfrak c}$, $C_\alpha\cap C_\beta=\omega$.
\end{proposition}
\begin{proof}
    Let $\theta$ be an uncountable regular cardinal such that $2^{\mathfrak c}<\theta$.
    
    Recursively, construct $(C_\alpha: \alpha<2^{\mathfrak c})$ and $(M_\alpha: \alpha<2^{\mathfrak c})$ such that for every $\alpha<2^{\mathfrak c}$:
    \begin{enumerate}[label=(\alph*)]
        \item $C_\alpha$ is a subspace of $\beta\omega$ containing $\omega$ of cardinality $\leq\mathfrak c$ which is $n$-cascade countably compact for every $n \in \omega$.
        \item $M_\alpha$ is a $\aleph_0$-closed elementary submodel of $H(\theta)$ of cardinality $\mathfrak c$ with $\beta\omega\in M_\alpha$.
        \item $\beta\omega\setminus \bigcup_{\beta<\alpha}(C_\beta\setminus \omega)\in M_\alpha$.
        \item $C_\alpha=M_\alpha \cap \left(\beta\omega\setminus \bigcup_{\beta<\alpha}(C_\beta\setminus \omega)\right)$.
    \end{enumerate}

        At each step of the recursion we proceed as in (b)--(d).
    To see that (a) holds, note that by (b) and (d), $C_\beta\subseteq M_\beta$ and $|M_\beta|=\mathfrak c$ for each $\beta<\alpha$, so
    \[
        \left|\bigcup_{\beta<\alpha}(C_\beta\setminus \omega)\right|
        \leq \max\{\alpha,\mathfrak c\}<2^{\mathfrak c}.
    \]
    Hence, by Proposition~\ref{proposition:betaOmegaSmallDimensional},
    $Y_\alpha=\beta\omega\setminus \bigcup_{\beta<\alpha}(C_\beta\setminus \omega)$
    is $n$-cascade countably compact for every $n \in \omega$.
    Since $Y_\alpha\in M_\alpha$ by (c), the same argument as in the proof of Theorem~\ref{theorem:betaOmegaHighDimensionalProduct} shows that
    $C_\alpha=M_\alpha\cap Y_\alpha$
    is $n$-cascade countably compact for every $n\in\omega$.

    Finally, if $\beta<\alpha$, then by (d),
    $C_\alpha\subseteq \beta\omega\setminus (C_\beta\setminus\omega)$,
    so $C_\alpha\cap C_\beta=\omega$.
\end{proof}

Now we are ready to improve \cite[Theorem~4.8]{corral2024highCompactness}.
As in \cite{corral2024highCompactness}, we proceed by modifying the spaces $X$ and $Y$ obtained in Theorem~\ref{theorem:betaOmegaHighDimensionalProduct} to create a single space whose square is not countably compact.
\begin{corollary}\label{corollary:betaOmegaHighDimensionalProduct}
    There exists a subspace of $\beta\omega$ containing $\omega$ whose square is not countably compact and that is $n$-cascade countably compact for every $n \in \omega$.
\end{corollary}
\begin{proof}
    Write $\omega=A\cup B$, where $A$ and $B$ are disjoint infinite subsets of $\omega$.
    
    Let $f:\omega\to A$ and $g:\omega\to B$ be bijections.
    Then $f, g$ extend to homeomorphisms $\bar f:\beta\omega\to \cl A$ and $\bar g:\beta\omega\to \cl B$.

    Let $X, Y$ be as in the previous theorem.
    Put $Z=\bar f[X]\cup \bar g[Y]$.
    We claim that $Z$ is as desired.

    To see that $Z$ is $n$-cascade countably compact for every $n \in \omega$, let $h:[\omega]^n\to Z$.
    By Ramsey's theorem, there exists $C \in [\omega]^{\omega}$ such that $h[[C]^n]$ is contained either in $\bar f[X]$ or in $\bar g[Y]$.
    Since $\bar f[X]$ and $\bar g[Y]$ are homeomorphic to $X$ and $Y$, respectively, they are $n$-cascade countably compact.
    Therefore, in either case, $h\restriction [C]^n$ has an $n$-cascade $p$-limit point in $Z$ for some free ultrafilter $p \in C^*$.

    Finally, the sequence $((f(n), g(n)): n \in \omega)$ has no accumulation points in $Z\times Z$.
    Suppose, toward a contradiction, that this sequence has an accumulation
    point. Since $\bar f[X]=Z\cap \cl A$ and $\bar g[Y]=Z\cap \cl B$, both
    $\bar f[X]$ and $\bar g[Y]$ are closed in $Z$.
    Hence, $\bar f[X]\times \bar g[Y]$ is closed in $Z\times Z$, so the sequence $((f(n), g(n)): n \in \omega)$ would have an accumulation point in $\bar f[X]\times \bar g[Y]$.
    Since $u=(\bar f,\bar g)$ is a homeomorphism from $X\times Y$ onto $\bar f[X]\times \bar g[Y]$, the sequence $((n,n):n\in\omega)$ would have an accumulation point in $X\times Y$, contrary to Theorem~\ref{theorem:betaOmegaHighDimensionalProduct}.
\end{proof}

This answers Question~4.15 of \cite{corral2024highCompactness} in ZFC. Moreover,
the space is $n$-cascade countably compact for every $n<\omega$, rather than
only $2$-cascade countably compact.

The final example separates ordinary barrier compactness from its recursive
cascade form.  It is the ZFC counterpart of
\cite[Theorem~4.10]{corral2024highCompactness}.

\begin{proposition}\label{proposition:betaOmega2NotCascade}
    There exists a subspace of $\beta\omega$ which is $k$-countably compact for every $k \in \omega$ but not $2$-cascade countably compact.
\end{proposition}

\begin{proof}
    By Proposition~\ref{proposition:betaOmegaManyHighDimensional}, there exists a sequence $(C_n: n \in \omega)$ of subspaces of $\beta\omega$ which are $k$-cascade countably compact for every $k \in \omega$ such that for every two distinct $n, n' \in \omega$, $C_n\cap C_{n'}=\omega$.

    For each $n\in \omega$, let $e_n:\omega\to \omega\times \omega$ be given by $e_n(m)=(n,m)$.
    Then $e_n$ has a unique extension to a topological embedding $\bar e_n:\beta\omega\to \beta(\omega\times \omega)$.
    Let $D_n=\bar e_n[C_n]$ for each $n \in \omega$.
    Thus, $D_n$ is homeomorphic to $C_n$ for each $n \in \omega$, and therefore is $k$-cascade countably compact, hence $k$-countably compact, for every $k \in \omega$.

    Let
    \[
        D_\omega=\{p \in \beta(\omega\times \omega): \{n\}\times \omega \notin p \text{ for all } n \in \omega\}
    \]
    and
    \[
        D=\bigcup_{\alpha\leq \omega} D_\alpha.
    \]
    Since $\omega\times\omega$ is countably infinite,
    $\beta(\omega\times\omega)$ is homeomorphic to $\beta\omega$. It is
    therefore enough to verify the two stated properties of $D$.

    First, we verify that $D$ is not $2$-cascade countably compact.
    Let $g:[\omega]^2\to D$ be given by $g(\{n,m\})=(n,m)$ for $n<m$.

    Fix a free ultrafilter $p\in\omega^*$.
    For each $n\in\omega$, since $\omega\setminus (n+1)\in p$, we have
    \[
    p\text{-}\lim_{m\in\omega\setminus (n+1)} g(\{n,m\})
    =
    p\text{-}\lim_{m\in\omega\setminus (n+1)} (n,m)
    =
    p\text{-}\lim_{m\in\omega\setminus (n+1)} \bar e_n(m)
    =
    \bar e_n(p).
    \]
    Since $\bar e_n(p)\in \cl(\{n\}\times\omega)$, if $\bar e_n(p)\in D$, then necessarily $\bar e_n(p)\in D_n$, and hence $p\in C_n$.
    As $C_n\cap C_m=\omega$ for $n\neq m$, a free ultrafilter belongs to at most one $C_n$.
    Therefore, for every free ultrafilter $p$, there is at most one $n\in\omega$ such that $p\text{-}\lim_{m\in\omega\setminus (n+1)} g(\{n,m\})\in D$.
    Since for $g$ to have a $2$-cascade $p$-limit point it would be necessary that
    $p\text{-}\lim_{m\in\omega\setminus (n+1)} g(\{n,m\})\in D$
    for $p$-many $n$, and $p$ is free, this is impossible.
    Thus, $g$ does not have a $2$-cascade $p$-limit point in $D$ for any free ultrafilter $p$, and so $D$ is not $2$-cascade countably compact.

    It remains to show that $D$ is $k$-countably compact for every $k\geq 1$.
    Let $k\geq 1$ and let $h:[\omega]^k\to D$.

    Suppose that, for some $n$, there is an infinite $A\subseteq\omega$ such
    that $h[[A]^k]\subseteq D_n$. Since $D_n$ is $k$-cascade countably
    compact, $h\restriction[A]^k$ has a $k$-cascade $p$-limit point in
    $D_n\subseteq D$ for some free ultrafilter $p$.

    Assume now that for every $n\in\omega$ and every $A\in[\omega]^{\omega}$ we have $h[[A]^k]\nsubseteq D_n$.
    Recursively construct a $\subseteq$-decreasing sequence $(A_n:n\in\omega)$ of infinite subsets of $\omega$ and a sequence $(m_n:n\in\omega)$ of natural numbers such that, for every $n\in\omega$:
    \begin{enumerate}[label=(\alph*)]
        \item $A_0=\omega$;
        \item $m_n=\min A_n$;
        \item $A_{n+1}\in[A_n\setminus\{m_n\}]^{\omega}$;
        \item $h[[A_{n+1}]^k]\cap D_n=\emptyset$.
    \end{enumerate}

    Indeed, once $A_n$ and $m_n$ are chosen, color $[A_n\setminus\{m_n\}]^k$ with two colors according to whether $h(s)\in D_n$ or not.
    By Ramsey's theorem, there is an infinite $A_{n+1}\in[A_n\setminus\{m_n\}]^{\omega}$ such that either $h[[A_{n+1}]^k]\subseteq D_n$ or $h[[A_{n+1}]^k]\cap D_n=\emptyset$.
    The first alternative is impossible by assumption, so (d) holds.

    Let $B=\{m_n: n \in \omega\}$.
    Fix a free ultrafilter $p$ on $\omega$ with $B\in p$.
    Since $p^{[B]^k}$ is an ultrafilter on $[B]^k$ and $\beta(\omega\times\omega)$ is compact, there exists
    \[
        x=p^{[B]^k}\text{-}\lim h \in \beta(\omega\times\omega).
    \]

    We claim that $x\in D_\omega$.
    For each $n\in\omega$, let
    \[
        U_n=\{y\in \beta(\omega\times\omega): \{n\}\times\omega\in y\}.
    \]
    $U_n$ is a clopen set.
    We show that $U_n$ is not a neighborhood of $x$.

    Let
    \[
    S_n=\{t\in [B]^k: h(t)\in U_n\}.
    \]
    Since $h(t)\in D$ for every $t\in[B]^k$, if $h(t)\in U_n$, then necessarily $h(t)\in D_n$.
    Now, if $t\in[B]^k$ and $\min(t)=m_j$ with $j>n$, then $t\subseteq A_{n+1}$, because $(A_i:i\in\omega)$ is decreasing and $\{m_j,m_{j+1},\dots\}\subseteq A_{n+1}$.
    Hence, by (d), $h(t)\notin D_n$.
    Therefore,
    \[
    S_n\subseteq \{t\in[B]^k:\min(t)\in\{m_0,\dots,m_n\}\}.
    \]
    The latter set does not belong to $p^{[B]^k}$, since $\{m_0,\dots,m_n\}\notin p$.
    Indeed, by the recursive definition of $p^{[B]^k}$, for every $F\subseteq B$,
    \[
        \{t\in [B]^k:\min(t)\in F\}\in p^{[B]^k}\iff F\in p.
    \]
    Thus, $S_n\notin p^{[B]^k}$, so $U_n$ is not a neighborhood of $x$.

    Since this holds for every $n\in\omega$, we have $\{n\}\times\omega\notin x$ for all $n\in\omega$, that is,
    \[
    x\in D_\omega\subseteq D.
    \]
    Therefore, $D$ is $k$-countably compact for every $k\geq 1$.

    This completes the proof.
\end{proof}

    \section{Products and \texorpdfstring{$n$}{n}-cascade countable compactness in groups}\label{section:product-groups}
In Corollary~\ref{corollary:betaOmegaHighDimensionalProduct}, we constructed a Tychonoff space $X$ that is $n$-cascade countably compact for every $n\in\omega$, whereas $X^2$ is not countably compact.
The next result shows that this phenomenon does not occur for Hausdorff topological groups.

Recall that a \emph{topological group} is a group $G$ endowed with a topology such that the group operation $G \times G \to G$ and the inversion map $G \to G$ are continuous.

\begin{theorem}\label{theorem:productOfCountablyCompactGroups}
    Let $G$ be a Hausdorff topological group, let $M\subseteq \omega$ be infinite, let $p\in\omega^*$ be such that $M\in p$, and let $1\le k<\omega$.
    Let $f=(f_1,\ldots,f_k):M\to G^k$.
    Define $h:[M]^k\to G$ by
    \[
        h(n_1,\ldots,n_k)=f_1(n_1)\cdot \ldots \cdot f_k(n_k)
    \]
    whenever $n_1<\cdots<n_k$ are in $M$.
    If $y\in G$ is a $[M]^k$-cascade $p$-limit point of $h$, then
    \[
        y=w_1\cdot \ldots \cdot w_k,
    \]
    where, for each $i=1,\ldots,k$, $w_i$ is a $p$-limit of $f_i$.

    Consequently, if $G$ is $k$-cascade countably compact, then $G^k$ is countably compact.
\end{theorem}

\begin{proof}
    We first prove the decomposition statement by induction on $k$.

    If $k=1$, then $h=f_1$ and the conclusion is immediate.

    Assume the statement holds for some $k\geq 1$, and let $f=(f_1,\ldots,f_{k+1}):M\to G^{k+1}$.
    Define $g:[M]^{k+1}\to G$ by
    \[
        g(n_1,\ldots,n_{k+1})=f_1(n_1)\cdot \ldots \cdot f_{k+1}(n_{k+1})
    \]
    whenever $n_1<\cdots<n_{k+1}$ are in $M$, and suppose that $y\in G$ is a $[M]^{k+1}$-cascade $p$-limit point of $g$.

    Then there exist $A\in p$, with $A\subseteq M$, and a family $(z_n:n\in A)$ in $G$ such that, for each $n\in A$, $z_n$ is a $[M\setminus(n+1)]^k$-cascade $p$-limit point of
    \[
        g_n:[M\setminus(n+1)]^k\to G,
        \qquad
        g_n(n_2,\ldots,n_{k+1})=f_1(n)\cdot f_2(n_2)\cdot \ldots \cdot f_{k+1}(n_{k+1}),
    \]
    and
    \[
        y=p\text{-}\lim_{n\in A} z_n.
    \]

    For each $n\in A$, let $h_n:[M\setminus(n+1)]^k\to G$ be given by
    \[
        h_n(n_2,\ldots,n_{k+1})=f_2(n_2)\cdot \ldots \cdot f_{k+1}(n_{k+1}).
    \]
    Since left translation by $f_1(n)^{-1}$ is continuous and sends $g_n$ to $h_n$, it follows that $f_1(n)^{-1}\cdot z_n$ is a $[M\setminus(n+1)]^k$-cascade $p$-limit point of $h_n$.

    As $M\setminus(n+1)\in p$ for every $n\in\omega$, the induction hypothesis
    applies. Thus, for each $n\in A$, there exist
    $w_2^n,\ldots,w_{k+1}^n\in G$ such that $w_i^n$ is a $p$-limit of
    $f_i\!\upharpoonright(M\setminus(n+1))$ for $i=2,\ldots,k+1$. Each
    $w_i^n$ is therefore also a $p$-limit of $f_i$, and
    \[
        f_1(n)^{-1}\cdot z_n = w_2^n\cdot \ldots \cdot w_{k+1}^n.
    \]

    Since $G$ is Hausdorff, $p$-limits are unique.
    Therefore, for each $i=2,\ldots,k+1$, the element $w_i^n$ does not depend on $n\in A$.
    Denote this common value by $w_i$, and set
    \[
        W=w_2\cdot \ldots \cdot w_{k+1}.
    \]
    Then, for every $n\in A$, $z_n=f_1(n)\cdot W$.
    Since $f_1(n)=z_n\cdot W^{-1}$ for every $n\in A$, and right translation by $W^{-1}$ is continuous, from $y=p\text{-}\lim_{n\in A} z_n$ we obtain
    \[
        p\text{-}\lim_{n\in A} f_1(n)=y\cdot W^{-1}.
    \]
    As $A\in p$, this is also the $p$-limit of $f_1$.
    Let $w_1:=y\cdot W^{-1}$.
    Therefore,
    \[
        y=w_1\cdot w_2\cdot \ldots \cdot w_{k+1},
    \]
    completing the induction.

    Now assume that $G$ is $k$-cascade countably compact, and let $f:\omega\to G^k$ be any sequence.
    Write $f=(f_1,\ldots,f_k)$, where each $f_i:\omega\to G$ is the $i$-th coordinate function of $f$.
    Define $h:[\omega]^k\to G$ by
    \[
        h(n_1,\ldots,n_k)=f_1(n_1)\cdot \ldots \cdot f_k(n_k)
    \]
    whenever $n_1<\cdots<n_k$.

    Since $G$ is $k$-cascade countably compact, there exist $p\in\omega^*$ and a $[\omega]^k$-cascade $p$-limit point $y\in G$ of $h$.
    By the first part of the theorem, there exist $w_1,\ldots,w_k\in G$ such that each $w_i$ is a $p$-limit of $f_i$.

    Since convergence with respect to $p$-limits in the product is coordinatewise, it follows that
    \[
        p\text{-}\lim_{n\in\omega} f(n)=(w_1,\ldots,w_k)
    \]
    in $G^k$.
    Hence, every sequence in $G^k$ has an accumulation point, so $G^k$ is countably compact.
\end{proof}

\begin{corollary}\label{corollary:productOfCountablyCompactGroups}
    Let $G$ be a Hausdorff topological group and let $\mathcal B$ be the Schreier barrier.
    If $G$ is $\mathcal B$-cascade countably compact, then $G^\omega$ is countably compact.
\end{corollary}

\begin{proof}
    Let $f:\omega\to G^\omega$ be a sequence, and write
    \[
        f(n)=(f_0(n),f_1(n),\ldots)
    \]
    for each $n\in\omega$.

    Define $h:\mathcal B\to G$ as follows.
    If $s=\{0\}$, let $h(s)=e_G$.
    If
    \[
        s=\{m<n_0<\cdots<n_{m-1}\}\in\mathcal B
    \]
    with $m\ge 1$, then
    \[
        h(s)=f_0(n_0)\cdot f_1(n_1)\cdot \ldots \cdot f_{m-1}(n_{m-1}).
    \]

    Since $G$ is $\mathcal B$-cascade countably compact, there exist $p\in\omega^*$, an infinite set $A\in p$, a family $(z_m:m\in A)$ in $G$, and $z\in G$ such that, for each $m\in A$, $z_m$ is a $\mathcal B_{\{m\}}$-cascade $p$-limit point of $h_m$, and
    \[
        z=p\text{-}\lim_{m\in A} z_m.
    \]

    Replacing $A$ with $A\setminus\{0\}$ if necessary, we may assume that $0\notin A$.
    For each $m\in A$, we have
    \[
        \mathcal B_{\{m\}}=[\omega\setminus(m+1)]^m.
    \]
    Since $m\ge 1$ and $\omega\setminus(m+1)\in p$, by Theorem~\ref{theorem:productOfCountablyCompactGroups}, for each $m\in A$ there exist $w_0^m,\ldots,w_{m-1}^m\in G$ such that, for every $i<m$, $w_i^m$ is a $p$-limit of $f_i$, and
    \[
        z_m=w_0^m\cdot \ldots \cdot w_{m-1}^m.
    \]

    Since $G$ is Hausdorff, $p$-limits are unique.  Thus, for each $i$, the
    value $w_i^m$ is independent of the choice of $m\in A$ with $i<m$.

    For each $i\in\omega$, choose $m\in A$ with $i<m$, and define
    \[
        w_i:=w_i^m.
    \]
    This is well-defined by the previous paragraph, and $w_i$ is a $p$-limit of $f_i$.

    It follows that
    \[
        p\text{-}\lim_{n\in\omega} f(n)=(w_i)_{i\in\omega} \in G^\omega.
    \]
    Thus, $G^\omega$ is countably compact.
\end{proof}

    \section{Two Boolean group constructions}\label{section:boolean-groups}
In this section, we construct two Hausdorff Boolean groups without non-trivial convergent
sequences. The first is $\mathcal B$-cascade countably compact for every
barrier $\mathcal B$, and its countable power is countably compact. The second
is $\mathcal B$-countably compact for every barrier $\mathcal B$, but its
square is not countably compact.

The constructions use cascade frames, families of ultrafilters, and separating
homomorphisms. The next four subsections develop these ingredients.
\subsection{Preliminaries}
Let $\kappa$ be a cardinal.
In this section, $G_\kappa$ denotes the Abelian group $[\kappa]^{<\omega}$, with group operation given by symmetric difference, $\ominus$, and identity element $\emptyset$.
Since $G_\kappa$ is a Boolean group, it can be seen as a vector space over $\mathbb Z_2$.

For every set $A$, $G_\kappa^A=\{f: f \text{ is a function from }A\text{ into } G_\kappa\}$ has the natural structure of a vector space over $\mathbb Z_2$, with operations defined coordinatewise.

The following is straightforward from linear algebra.
\begin{lemma}\label{lemma:linearAlgebra}
    The family $(\{\xi\}: \xi \in \kappa)$ is a basis of $G_\kappa$.
    Thus, for every function $f: \kappa \to V$ into a $\mathbb Z_2$-vector space $V$, there exists a unique linear mapping $\phi: G_\kappa \to V$ such that $\phi(\{\xi\})=f(\xi)$ for every $\xi \in \kappa$.
\end{lemma}

Below, we introduce notation for the quotient of $G_\kappa^A$ by the subspace of eventually constant functions.
\begin{definition}
    The space $\finct(\kappa, A)$ is the quotient vector space of $G_\kappa^A$ by the subspace
    \[
        \{f \in G_\kappa^A: \exists c \in G_\kappa\, \exists F \in [A]^{<\omega}\,
        \forall n \in A\setminus F,\ f(n)=c\}.
    \]
\end{definition}

If $A, B, C$ are sets with $A\subseteq B$, let $\Fun(A, B, C)$ denote the set of all functions $f$ such that $A\subseteq \dom(f)\subseteq B$ and $\ran(f)\subseteq^* C$.

\begin{definition}
    For $f\in \Fun(A, \omega, G_\kappa)$, we write $[f]_{A}^\kappa$ for the equivalence class of $\bar f$ in $\finct(\kappa, A)$, where $\bar f$ is any function in $G_\kappa^A$ such that $\bar f(n)=f(n)$ for all but finitely many $n \in A$.

    When $\kappa$ is clear from the context, we simply write $[f]_A$.

    We say that a family $(f_\xi:\, \xi \in I)$ in $\Fun(A, \omega, G_{\kappa})$ is \emph{linearly independent mod fin constants in $A$} ($\li(A)$, for short) if $([f_\xi]_{A}:\, \xi \in I)$ is linearly independent in $\finct(\kappa, A)$.
\end{definition}
Thus, in the previous notation, $(f_\xi:\, \xi \in I)$ is $\li(A)$ if and only if for every nonempty finite $F\subseteq I$ and every $c\in G_\kappa$, the set $\{n\in A:\, \bigominus_{\xi \in F} f_\xi(n)\neq c\}$ is infinite.

Notice that if $B\subseteq A$ is infinite, it may happen that $(f_\xi:\, \xi \in I)$ is $\li(A)$ but not $\li(B)$.
Because we repeatedly pass to infinite subsets, we use a stronger version of
linear independence mod finite constants that is preserved by such
restrictions.

\begin{definition}
    Let $A\in [\omega]^{\omega}$.
    We say that a family $(g_\beta:\, \beta \in I)$ in $\Fun(A, \omega, G_{\kappa})$ is \emph{injectively linearly independent mod fin in $A$} ($\LI(A)$, for short) if there exists a family $(F_\beta: \beta \in I)$ of finite subsets of $A$ such that the family $(g_\beta(m):\beta \in I, m \in A\setminus F_\beta)$ is linearly independent in $G_\kappa$.
\end{definition}

The following lemma is immediate from the definitions.
\begin{lemma}
Let $\kappa$ be an infinite cardinal, $A\subseteq \omega$ be infinite and $(g_\alpha:\, \alpha \in I)$ be a family in $\Fun(A, \omega, G_{\kappa})$.
Then:
\begin{enumerate}[label=(\alph*)]
    \item If $(g_\alpha:\, \alpha \in I)$ is $\LI(A)$, then $(g_\alpha:\, \alpha \in I)$ is $\li(A)$.
    \item If $B\subseteq^* C\subseteq^* A$ are infinite, and $(g_\alpha:\, \alpha \in I)$ is $\LI(C)$, then $(g_\alpha:\, \alpha \in I)$ is $\LI(B)$.
\end{enumerate}
\end{lemma}

The next lemma extracts, after restriction to an infinite subset, an
injectively linearly independent family whose equivalence classes span those
of the original functions.

\begin{lemma}\label{lemma:LItoLI}
Let $A\subseteq \omega$ be infinite and $I, J$ be disjoint countable sets, $(g_\alpha:\alpha \in I)$ be $\LI(A)$ and $(g_\beta:\, \beta \in J)$ be a family in $\Fun(A, \omega, G_{\kappa})$.

Then there exists an infinite $C\subseteq A$ and $J'\subseteq J$ such that $(g_\alpha|_C:\, \alpha \in I\cup J')$ is $\LI(C)$ and such that, for every $\beta \in J$, $[g_\beta]_C$ is a linear combination of $\{[g_\alpha]_{C}:\, \alpha \in I\cup J'\}$.
\end{lemma}
\begin{proof}
    Enumerate $[I\cup J]^{<\omega}\setminus\{\emptyset\}$ as $(F_i:i\in\omega)$.
    Recursively define a decreasing sequence $(B_i:i\in\omega)$ of infinite subsets of $A$
    such that $B_0=A$ and, for every $i\in\omega$,
    \[
        \bigominus_{\beta\in F_i} g_\beta|_{B_{i+1}} \text{ is either constant or injective.}
    \]

    This is possible because every function defined on an infinite subset of $\omega$
    has an infinite restriction which is either constant or injective.
    Let $B$ be a pseudointersection of $(B_i:i\in\omega)$.

    Since $(g_\alpha:\alpha\in I)$ is $\LI(A)$ and $B\subseteq^* A$, it follows from the
    previous lemma that $(g_\alpha|_B:\alpha\in I)$ is $\LI(B)$. Hence,
    $(g_\alpha|_B:\alpha\in I)$ is $\li(B)$, so $([g_\alpha]_B:\alpha\in I)$ is linearly
    independent in $\finct(\kappa,B)$.

    Therefore, there exists a subset $J'\subseteq J$ such that
    $([g_\alpha]_B:\alpha\in I\cup J')$ is linearly independent and spans the same
    subspace of $\finct(\kappa,B)$ as $([g_\alpha]_B:\alpha\in I\cup J)$.
    In particular, for every $\beta\in J$, the class $[g_\beta]_B$ is a linear
    combination of $\{[g_\alpha]_B:\alpha\in I\cup J'\}$.

    We claim that for every nonempty finite set $F\subseteq I\cup J'$, there exists
    a finite set $L\subseteq B$ such that
    \[
        \bigominus_{\beta\in F} g_\beta|_{B\setminus L} \text{ is injective.}
    \]

    Indeed, let $F\subseteq I\cup J'$ be nonempty and finite, and choose $i\in\omega$
    such that $F=F_i$. Since $B\subseteq^* B_{i+1}$, there exists a finite set
    $L\subseteq B$ such that $B\setminus L\subseteq B_{i+1}$.
    By construction, $\bigominus_{\beta\in F} g_\beta|_{B_{i+1}}$ is either constant or injective.
    It cannot be constant, for otherwise $\bigominus_{\beta\in F} [g_\beta]_B=0$ in $\finct(\kappa,B)$, contradicting the linear independence of
    $([g_\alpha]_B:\alpha\in I\cup J')$.
    Hence, $\bigominus_{\beta\in F} g_\beta|_{B\setminus L}$ is injective, as claimed.

    Enumerate $I\cup J'$ as $(\beta_n:n\in\omega)$.
    Recursively choose a sequence $(m_n:n\in\omega)$ of pairwise distinct elements of $B$
    as follows. Suppose $m_0,\dots,m_{n-1}$ have already been chosen. Let
    \[
        V_n=\left\langle g_{\beta_k}(m_j): j<n,\ k\le j\right\rangle.
    \]
    Since $V_n$ is a finite-dimensional $\mathbb Z_2$-subspace of $G_\kappa$, it is finite.

    For each nonempty $K\subseteq n+1$, by the claim above there exists a finite set
    $L_K\subseteq B$ such that $\bigominus_{k\in K} g_{\beta_k}|_{B\setminus L_K}$ is injective.
    Therefore,
    \[
        S_K:=\left\{m\in B:\bigominus_{k\in K} g_{\beta_k}(m)\in V_n\right\} \text{ is finite.}
    \]
    Since there are only finitely many subsets $K$ of $n+1$, the set $\{m_0,\dots,m_{n-1}\}\cup\bigcup_{\emptyset\neq K\subseteq n+1} S_K$ is finite.
    Choose
    \[
        m_n\in B\setminus\left(\{m_0,\dots,m_{n-1}\}\cup
        \bigcup_{\emptyset\neq K\subseteq n+1} S_K\right).
    \]

    Let $C=\{m_n:n\in\omega\}\subseteq B$. For each $n\in\omega$, define $E_{\beta_n}=\{m_j:j<n\}$.
    We claim that $(g_\alpha|_C:\alpha\in I\cup J')$ is $\LI(C)$, as witnessed by
    $(E_\beta:\beta\in I\cup J')$.

    Indeed, this means exactly that
    \[
        (g_{\beta_n}(m_j): n\in\omega,\ j\in\omega,\ n\le j) \text{ is linearly independent in } G_\kappa.
    \]
    Suppose, towards a contradiction, that $\bigominus_{(n,j)\in T} g_{\beta_n}(m_j)=0$ for some nonempty finite set $T\subseteq\{(n,j)\in\omega^2:n\le j\}$.
    Let $j^*=\max\{j:\exists n\ ((n,j)\in T)\}$ and let
    \[
        K=\{n:(n,j^*)\in T\}.
    \]
    Then $K$ is a nonempty subset of $j^*+1$, and
    \[
        \bigominus_{n\in K} g_{\beta_n}(m_{j^*})
        =
        \bigominus_{(n,j)\in T,\ j<j^*} g_{\beta_n}(m_j)
        \in V_{j^*}.
    \]
    This contradicts the choice of $m_{j^*}$.
    Therefore, $(g_\alpha|_C:\alpha\in I\cup J')$ is $\LI(C)$.

    Finally, since for every $\beta\in J$ the class $[g_\beta]_B$ is a linear combination
    of $\{[g_\alpha]_B:\alpha\in I\cup J'\}$, by restricting from $B$ to $C$ we obtain
    that $[g_\beta]_C$ is a linear combination of
    $\{[g_\alpha]_C:\alpha\in I\cup J'\}$ for every $\beta\in J$.
\end{proof}

\subsection{Cascade frames}
In this subsection, we introduce the notion of \emph{cascade frames}, which are structures that allow us to build cascade limits in a controlled manner.

\begin{definition}\label{definition:cascadeFrame}
Let $A \in [\omega]^{\omega}$, $B \in [A]^\omega$, $I\subseteq \kappa$, $\mathcal B$ be a barrier on $A$, 
$f:\mathcal B \to G_\kappa$, $(g_\beta:\beta \in I)$ be a family in 
$\Fun(A,\omega,G_\kappa)$, and $(x_t:t \in \mathcal T_{\mathcal B})$ be a family in $G_\kappa$.

Recursively on the rank of $\mathcal B$, we say that 
$(B,(g_\beta:\beta \in I),(x_t:t \in \mathcal T_{\mathcal B}))$ is a 
\emph{$\mathcal B$-cascade frame} for $f$ if:
\begin{enumerate}[label=(\alph*)]
    \item $(g_\beta:\beta \in I)$ is $\LI(B)$.\label{item:definition:cascadeFrame-LI}
    \item For each $t \in \mathcal B$, $x_t=f(t)$.
    \item\label{item:definition:cascadeFrame-positiveRank} If $\mathrm{rk}(\mathcal B)>0$, then there exists a finite set
    $F\subseteq I$ and $Z\in G_\kappa$ such that
    \[
    x_{\{n\}}=\bigominus_{\beta\in F} g_\beta(n)\ominus Z
    \]
    for all but finitely many $n\in B$, and $x_\emptyset=F\ominus Z$.
    In particular, $[(x_{\{n\}}:n\in B)]_B$ is a linear combination of 
    $\{[g_\beta]_B:\beta\in I\}$ in $\finct(\kappa,B)$.
    \item\label{item:definition:cascadeFrame-recursive} If $\mathrm{rk}(\mathcal B)>0$, then for each $n \in B$,
    \[
    (B\setminus(n+1),(g_\beta:\beta\in I),(x_{\{n\}\cup t}:t\in \mathcal T_{\mathcal B_{\{n\}}}))
    \]
    is a $\mathcal B_{\{n\}}$-cascade frame for $f_n$.
    \item $g_\beta(n)\in [\beta]^{<\omega}$ for each $\beta \in I$ and each $n \in A$.
\end{enumerate}
\end{definition}

Before showing that such objects exist, we explain the main reason for introducing them.

\begin{lemma}\label{lemma:cascadeFrameLimit}
Let $A \in [\omega]^{\omega}$, let $\mathcal B$ be a barrier on $A$, let
$f:\mathcal B \to G_\kappa$, let $(g_\beta:\beta \in I)$ be a family in
$\Fun(A,\omega,G_\kappa)$, and let
$(B,(g_\beta:\beta \in I),(x_t:t \in \mathcal T_{\mathcal B}))$
be a $\mathcal B$-cascade frame for $f$.

Assume $G_\kappa$ has a group topology for which there exists a free ultrafilter $q$ on $\omega$ such that:
\begin{enumerate}[label=(\arabic*)]
    \item $B \in q$, and
    \item for every $\beta \in I$, $\{\beta\}\in G_\kappa$ is a $q$-limit of $g_\beta$.
\end{enumerate}
Then $x_\emptyset$ is a $\mathcal B$-cascade $q$-limit of $f$.
\end{lemma}

\begin{proof}
We proceed by induction on the rank of $\mathcal B$.
If the rank of $\mathcal B$ is $0$, then $f(\emptyset)=x_\emptyset$ and we are done.

Otherwise, for each $m \in B$, by
\ref{item:definition:cascadeFrame-recursive} of the definition,
$(B\setminus(m+1),(g_\beta:\beta\in I),(x_{\{m\}\cup t}:t\in \mathcal T_{\mathcal B_{\{m\}}}))$
is a $\mathcal B_{\{m\}}$-cascade frame for $f_m$.
Since $B\in q$ and $q$ is free, we have $B\setminus(m+1)\in q$.
Hence, by induction hypothesis, $x_{\{m\}}$ is a $\mathcal B_{\{m\}}$-cascade $q$-limit of $f_m$.

By \ref{item:definition:cascadeFrame-positiveRank} of the definition, write $[(x_{\{n\}}: n \in B)]_B$ as $\bigominus_{\beta \in F}[g_\beta]_{B}$.
There exists $Z\in G_\kappa$ such that for all but finitely many $n \in B$, $x_{\{n\}}=\bigominus_{\beta \in F} g_\beta(n)\ominus Z$.
Thus, $q\text{-}\lim_{n \in B} x_{\{n\}}=\bigominus_{\beta \in F} q\text{-}\lim_{n \in B} g_\beta(n)\ominus Z=\bigominus_{\beta \in F} \{\beta\}\ominus Z=F\ominus Z=x_\emptyset$.
Therefore, by definition, $x_\emptyset$ is a $\mathcal B$-cascade $q$-limit of $f$.
\end{proof}

\begin{corollary}\label{corollary:cascadeFrameLimit}
    Let $A \in [\omega]^{\omega}$, let $\mathcal B$ be a barrier on $A$, let
    $f:\mathcal B \to G_\kappa$, let $(g_\beta:\beta \in I)$ be a family in
    $\Fun(A,\omega,G_\kappa)$, and let
    $(B,(g_\beta:\beta \in I),(x_t:t \in \mathcal T_{\mathcal B}))$
    be a $\mathcal B$-cascade frame for $f$.
    Let $\phi:\omega\to B$ be injective, and let $p\in\omega^*$.

    If for every $\beta \in I$, $\{\beta\}\in G_\kappa$ is a $p$-limit of $g_\beta\circ \phi$, then $x_\emptyset$ is a $\mathcal B$-cascade $q$-limit of $f$ for some $q\in \omega^*$.
\end{corollary}

\begin{proof}
Let $q$ be the free ultrafilter on $\omega$ defined by $X\in q$ if and only if $\phi^{-1}[X]\in p$.
Then, given an open neighborhood $W$ of $\{\beta\}$, $\phi^{-1}[g_\beta^{-1}[W]]\in p$, so $g_\beta^{-1}[W]\in q$.
Thus, $\{\beta\}$ is a $q$-limit of $g_\beta$ for every $\beta \in I$.
Also, $B \in q$ since $\phi^{-1}[B]=\omega \in p$.
Thus, the claim follows from Lemma~\ref{lemma:cascadeFrameLimit}.
\end{proof}

The following technical lemma is used in the inductive step of the proof of Theorem~\ref{theorem:cascadeFrameExistence} to construct $\mathcal B$-cascade frames.

\begin{lemma}\label{lemma:completeRootOfCascadeFrame}
Let $\kappa$ be an infinite cardinal of uncountable cofinality, let $L$ be a cofinal subset of $\kappa$, let $A\in[\omega]^{\omega}$, let $\mathcal B$ be a barrier on $A$ of positive rank, let $H\subseteq \kappa$ be countable, let $D\in[A]^{\omega}$, let $f:\mathcal B\to G_\kappa$, let $(g_\beta:\beta\in H)$ be a family in $\Fun(A,\omega,G_\kappa)$, and let $(x_t:t\in\mathcal T_{\mathcal B}\setminus\{\emptyset\})$
be a family in $G_\kappa$.

Assume that:
\begin{enumerate}[label=(\roman*)]
    \item $(g_\beta|_D:\beta\in H)$ is $\LI(D)$;
    \item for every $n\in D$,
    \[
    (D\setminus(n+1),(g_\beta:\beta\in H),(x_{\{n\}\cup t}:t\in\mathcal T_{\mathcal B_{\{n\}}}))
    \]
    is a $\mathcal B_{\{n\}}$-cascade frame for $f_n$;
    \item for every $t\in\mathcal T_{\mathcal B}\setminus\{\emptyset\}$,
    \[
    x_t\in \langle f(b):b\in\mathcal B\rangle \ominus [H]^{<\omega}.
    \]
\end{enumerate}

Then there exist an infinite set $B\subseteq D$, a countable set $J\subseteq L$ disjoint from $H$, a family $(g_\beta:\beta\in J)$ in $\Fun(A,\omega,G_\kappa)$, and a point $x_\emptyset\in G_\kappa$ such that
\[
(B,(g_\beta:\beta\in H\cup J),(x_t:t\in\mathcal T_{\mathcal B}))
\]
is a $\mathcal B$-cascade frame for $f$, and
$x_t\in \langle f(b):b\in\mathcal B\rangle \ominus [H\cup J]^{<\omega}$
for every $t\in\mathcal T_{\mathcal B}$.
\end{lemma}

\begin{proof}
As $\cf(\kappa)>\omega$ and $L$ is cofinal in $\kappa$, we may choose $\gamma\in L\setminus H$
such that $x_{\{n\}}\subseteq \gamma$ for every $n\in D$.
Define $g_\gamma\in \Fun(A,\omega,G_\kappa)$ by
\[
g_\gamma(n)=
\begin{cases}
x_{\{n\}}, & n\in D,\\[2mm]
\emptyset, & n\in A\setminus D.
\end{cases}
\]
Then $g_\gamma(n)\in[\gamma]^{<\omega}$ for every $n\in A$.

Apply Lemma~\ref{lemma:LItoLI} to the $\LI(D)$ family
$(g_\beta|_D:\beta\in H)$
and the one-element family $(g_\gamma|_D)$.
Then there exists an infinite set $B\subseteq D$ such that exactly one of the following two cases holds:

\begin{enumerate}[label=(\textbf{Case \arabic*:}), leftmargin=*]
    \item $[g_\gamma]_B$ is a linear combination of
    $\{[g_\beta]_B:\beta\in H\}$;
    \item $(g_\beta|_B:\beta\in H\cup\{\gamma\})$ is $\LI(B)$.
\end{enumerate}

In Case~1, choose a finite set $F\subseteq H$ and $Z\in G_\kappa$ such that
\[
g_\gamma(n)=\bigominus_{\beta\in F}g_\beta(n)\ominus Z
\]
for all but finitely many $n\in B$.
Let $J=\emptyset$ and define
$x_\emptyset=F\ominus Z$.
Then clause \ref{item:definition:cascadeFrame-positiveRank} of the definition of $\mathcal B$-cascade frame holds for $B$.

We now verify that
$x_\emptyset\in \langle f(b):b\in\mathcal B\rangle \ominus [H]^{<\omega}$.

Choose $n\in B$ such that
\[
x_{\{n\}}=g_\gamma(n)=\bigominus_{\beta\in F}g_\beta(n)\ominus Z.
\]
By (iii),
$x_{\{n\}}\in \langle f(b):b\in\mathcal B\rangle \ominus [H]^{<\omega}$,
so there exist $h\in \langle f(b):b\in\mathcal B\rangle$ and $s\in [H]^{<\omega}$ such that
$x_{\{n\}}=h\ominus s$.
Since each $g_\beta(n)\in[\beta]^{<\omega}$, we have
\[
u:=\bigominus_{\beta\in F}g_\beta(n)\in [H]^{<\omega}.
\]
Thus,
$h\ominus s=u\ominus Z$,
hence,
\[
Z=h\ominus s\ominus u\in
\langle f(b):b\in\mathcal B\rangle \ominus [H]^{<\omega}.
\]
Since $F\in[H]^{<\omega}$, it follows that
\[
x_\emptyset=F\ominus Z\in
\langle f(b):b\in\mathcal B\rangle \ominus [H]^{<\omega}.
\]

In Case~2, let $J=\{\gamma\}$ and define
$x_\emptyset=\{\gamma\}$.
Then clause \ref{item:definition:cascadeFrame-positiveRank} holds for $B$, with $F=\{\gamma\}$ and $Z=\emptyset$.
Moreover,
\[
x_\emptyset=\{\gamma\}\in [H\cup J]^{<\omega}\subseteq
\langle f(b):b\in\mathcal B\rangle \ominus [H\cup J]^{<\omega}.
\]

In either case, for every $n\in B$,
$(B\setminus(n+1),(g_\beta:\beta\in H\cup J),(x_{\{n\}\cup t}:t\in\mathcal T_{\mathcal B_{\{n\}}}))$
is a $\mathcal B_{\{n\}}$-cascade frame for $f_n$:
this follows from (ii) by restricting from $D\setminus(n+1)$ to the infinite subset $B\setminus(n+1)$; clause \ref{item:definition:cascadeFrame-LI} follows from the fact that $(g_\beta:\beta\in H\cup J)$ is $\LI(B\setminus(n+1))$ (by heredity in Case~1, and by construction in Case~2), and the support condition for $\beta=\gamma$ was verified above.

Therefore,
\[
(B,(g_\beta:\beta\in H\cup J),(x_t:t\in\mathcal T_{\mathcal B}))
\]
is a $\mathcal B$-cascade frame for $f$.

Finally,
$x_t\in \langle f(b):b\in\mathcal B\rangle \ominus [H\cup J]^{<\omega}$
holds for every $t\in\mathcal T_{\mathcal B}$: this was verified above for
$t=\emptyset$, and for every non-root node it follows from (iii), since $[H]^{<\omega}\subseteq [H\cup J]^{<\omega}$.
\end{proof}

In the next theorem, we show how to construct cascade frames.

\begin{theorem} Let $\kappa$ be an infinite cardinal of uncountable cofinality, $L$ be a cofinal subset of $\kappa$, $I\subseteq \kappa$ be countable, $A \in [\omega]^{\omega}$, $\mathcal B$ be a barrier on $A$, $(g_\beta:\, \beta \in I)$ be $\LI(A)$ and $f:\mathcal B\to G_\kappa$. \label{theorem:cascadeFrameExistence}

Then there exist a countable set $J$ disjoint from $I$, an infinite set $B\subseteq A$, a family $(g_\beta:\beta\in J)$, and a family $(x_t:t\in\mathcal T_{\mathcal B})$ such that 

\begin{itemize}
\item $(B,(g_\beta:\, \beta \in I\cup J), (x_t:\, t \in \mathcal T_{\mathcal B}))$ is a $\mathcal B$-cascade frame for $f$.
\item $J\subseteq L$.
\item for every $t \in \mathcal T_{\mathcal B}$, $x_t \in \langle f(b): b \in \mathcal B\rangle \ominus [I\cup J]^{<\omega}$.
\end{itemize}
\end{theorem}

\begin{proof}
    We proceed by induction on the rank of $\mathcal B$.

    If $\mathrm{rk}(\mathcal B)=0$, then $\mathcal B=\{\emptyset\}$.
    Let $J=\emptyset$, let $(g_\beta:\beta\in J)$ be arbitrary, and define
    $x_\emptyset=f(\emptyset)$ and $B=A$.
    Then the conclusion is immediate.

    Assume now that $\mathrm{rk}(\mathcal B)>0$ and that the theorem holds for all barriers of smaller rank.

    Fix pairwise disjoint infinite sets
    $S_{-1},S_0,S_1,\dots \subseteq \omega$
    such that $\min S_m>m$ for every $m\in\omega$.
    Choose an injection
    $\sigma_I:I\to S_{-1}$.

    We recursively construct infinite sets $A_m\subseteq A$, points $k_m\in A_m$, countable sets $J_m\subseteq L$, functions $(g_\beta:\beta\in J_m)$ in $\Fun(A,\omega,G_\kappa)$, families $(x_t^m:t\in\mathcal T_{\mathcal B_{\{k_m\}}})$, injections $\sigma_m:J_m\to S_m$;
    such that, writing
    \[
    K_m=I\cup\bigcup_{i<m}J_i,
    \]
    the following hold for every $m\in\omega$:
    \begin{enumerate}[label=(\roman*)]
        \item $A_0=A$;
        \item $A_{m+1}\subseteq A_m\setminus (k_m+1)$;
        \item $(g_\beta:\beta\in K_m)$ is $\LI(A_m)$;
        \item
        $(A_{m+1},(g_\beta:\beta\in K_m\cup J_m),(x_t^m:t\in\mathcal T_{\mathcal B_{\{k_m\}}}))$
        is a $\mathcal B_{\{k_m\}}$-cascade frame for $f_{k_m}$;
        \item for every $t\in\mathcal T_{\mathcal B_{\{k_m\}}}$,
        $x_t^m\in
        \langle f(b):b\in\mathcal B\rangle
        \ominus [K_m\cup J_m]^{<\omega}$.
    \end{enumerate}

    We now explain the recursive step.
    Suppose $A_m$ and the family $(g_\beta:\beta\in K_m)$ have already been constructed, and that $(g_\beta:\beta\in K_m)$ is $\LI(A_m)$.

    Let
    \[
    \sigma=\sigma_I\cup\bigcup_{i<m}\sigma_i
    \]
    be the injection defined on $K_m$, and let
    \[
    E_m=\{\beta\in K_m:\sigma(\beta)\le m\}.
    \]
    Since $\sigma[K_m]\subseteq S_{-1}\cup\bigcup_{i<m}S_i$, each set $S_i\cap(m+1)$ is finite, and $\sigma$ is injective, the set $E_m$ is finite.

    For each $\beta\in E_m$, let $F_\beta^m\subseteq A_m$ be a finite witness for $\LI(A_m)$ restricted to the finite family $(g_\beta:\beta\in E_m)$.
    Thus, the family
    \[
    (g_\beta(n):\beta\in E_m,\ n\in A_m\setminus F_\beta^m)
    \]
    is linearly independent in $G_\kappa$.

    Let
    \[
    V_m=\left\langle g_\beta(k_j): j<m,\ \sigma(\beta)\le j\right\rangle.
    \]
    Since only finitely many pairs $(\beta,j)$ occur in the definition of $V_m$, the space $V_m$ is finite-dimensional, hence finite.

    For each nonempty finite set $H\subseteq E_m$, define
    \[
    h_H(n)=\bigominus_{\beta\in H} g_\beta(n)
    \qquad (n\in A_m).
    \]
    Because the family $(g_\beta:\beta\in E_m)$ is $\LI(A_m)$, the map $h_H$ is injective on $A_m\setminus \bigcup_{\beta\in H}F_\beta^m$.
    Therefore,
    \[
    M_H:=\{n\in A_m:h_H(n)\in V_m\}
    \]
    is finite.
    Since there are only finitely many nonempty subsets of $E_m$, the set
    \[
    \bigcup_{\emptyset\neq H\subseteq E_m}M_H
    \]
    is finite.
    Choose
    \[
    k_m\in A_m\setminus \bigcup_{\emptyset\neq H\subseteq E_m}M_H.
    \]

    Now $(g_\beta:\beta\in K_m)$ is $\LI(A_m\setminus(k_m+1))$ by heredity of $\LI$.
    Apply the inductive hypothesis to the barrier $\mathcal B_{\{k_m\}}$ on $A_m\setminus(k_m+1)$, to the family $(g_\beta:\beta\in K_m)$, and to the map $f_{k_m}$.
    We obtain a countable set $J_m\subseteq L$ disjoint from $K_m$, an infinite set
    \[
    A_{m+1}\subseteq A_m\setminus(k_m+1),
    \]
    functions $(g_\beta:\beta\in J_m)$ defined on $A_m\setminus(k_m+1)$, and a family
    \[
    (x_t^m:t\in\mathcal T_{\mathcal B_{\{k_m\}}})
    \]
    such that (iv) and (v) hold.

    For each $\beta\in J_m$, extend the function given by the inductive hypothesis arbitrarily to a function in $\Fun(A,\omega,G_\kappa)$; for instance, let
    $g_\beta(n)=\emptyset$ for every $n\in A\setminus (A_m\setminus(k_m+1))$.

    Finally, choose any injection $\sigma_m:J_m\to S_m$.
    This completes the recursive construction.
    Let
    \[
    J'=\bigcup_{m\in\omega}J_m,
    \qquad
    D=\{k_m:m\in\omega\},
    \qquad\text{and}\qquad
    \sigma=\sigma_I\cup\bigcup_{m\in\omega}\sigma_m.
    \]

    For each $\beta\in I\cup J'$, define
    \[
    F_\beta=\{k_j\in D:j<\sigma(\beta)\}.
    \]
    We claim that $(g_\beta|_D:\beta\in I\cup J')$ is $\LI(D)$, as witnessed by $(F_\beta:\beta\in I\cup J')$.

    Indeed, suppose towards a contradiction that
    \[
    \bigominus_{(\beta,j)\in T}g_\beta(k_j)=0
    \]
    for some nonempty finite set
    \[
    T\subseteq \{(\beta,j):\beta\in I\cup J',\ j\in\omega,\ j\ge \sigma(\beta)\}.
    \]

    Let
    $j^*=\max\{j:\exists\beta\,((\beta,j)\in T)\}$
    and
    $H=\{\beta:(\beta,j^*)\in T\}$.
    Then $H$ is a nonempty finite subset of $E_{j^*}$.
    Moreover,
    \[
    \bigominus_{\beta\in H}g_\beta(k_{j^*})
    =
    \bigominus_{(\beta,j)\in T,\ j<j^*}g_\beta(k_j)
    \in V_{j^*},
    \]
    contradicting the choice of $k_{j^*}$.
    This proves the claim.

    We now define $(x_t:t\in\mathcal T_{\mathcal B}\setminus\{\emptyset\})$.
    For each $m\in\omega$ and each $t\in\mathcal T_{\mathcal B_{\{k_m\}}}$, define
    \[
    x_{\{k_m\}\cup t}=x_t^m.
    \]
    If $t\in\mathcal B$ and $x_t$ has not yet been defined, let $x_t=f(t)$.
    For every remaining $t\in\mathcal T_{\mathcal B}\setminus\{\emptyset\}$ not yet defined, let $x_t=\emptyset$.

    We claim that for every $n=k_m\in D$,
    \[
    (D\setminus(n+1),(g_\beta:\beta\in I\cup J'),(x_{\{n\}\cup t}:t\in\mathcal T_{\mathcal B_{\{n\}}}))
    \]
    is a $\mathcal B_{\{n\}}$-cascade frame for $f_n$.
    Indeed, by construction,
    \[
    (A_{m+1},(g_\beta:\beta\in K_m\cup J_m),(x_t^m:t\in\mathcal T_{\mathcal B_{\{k_m\}}}))
    \]
    is a $\mathcal B_{\{k_m\}}$-cascade frame for $f_{k_m}$, and
    \[
    D\setminus(k_m+1)\subseteq A_{m+1}.
    \]
    Passing from $A_{m+1}$ to the infinite subset $D\setminus(k_m+1)$ preserves the defining clauses of a cascade frame; clause \ref{item:definition:cascadeFrame-LI} follows from the fact that $(g_\beta|_D:\beta\in I\cup J')$ is $\LI(D)$ and $\LI$ is hereditary, while enlarging the index family from $K_m\cup J_m$ to $I\cup J'$ does not affect the remaining clauses.

    Moreover, for every non-root node $t\in\mathcal T_{\mathcal B}\setminus\{\emptyset\}$,
    \[
    x_t\in \langle f(b):b\in\mathcal B\rangle \ominus [I\cup J']^{<\omega},
    \]
    because this holds for all values coming from the inductive construction, and it is trivial for the remaining values, which were defined either as $f(t)$ or as $\emptyset$.

    Apply Lemma~\ref{lemma:completeRootOfCascadeFrame} with
    $H=I\cup J'$.
    We obtain an infinite set $B\subseteq D$, a countable set $J''\subseteq L$ disjoint from $I\cup J'$, a family $(g_\beta:\beta\in J'')$, and a point $x_\emptyset\in G_\kappa$ such that
    \[
    (B,(g_\beta:\beta\in I\cup J'\cup J''),(x_t:t\in\mathcal T_{\mathcal B}))
    \]
    is a $\mathcal B$-cascade frame for $f$, and
    \[
    x_t\in \langle f(b):b\in\mathcal B\rangle \ominus [I\cup J'\cup J'']^{<\omega}
    \]
    for every $t\in\mathcal T_{\mathcal B}$.

    Let
    $J=J'\cup J''$.
    Then $J$ is countable, $J\subseteq L$, and $J\cap I=\emptyset$.
    This concludes the proof.
\end{proof}

\subsection{Tidy families of ultrafilters}

In \cite{hruvsak2021countably}, a family of ultrafilters with the following property was constructed.

\begin{proposition}
    There exists a family of ultrafilters on $\omega$, $(p_\alpha:\alpha<\mathfrak c)$, such that for every countable $D\subseteq \mathfrak c$ and every family $(h_\alpha:\alpha\in D)$, where each $h_\alpha$ is an injective enumeration of a linearly independent subset of $[\mathfrak c]^{<\omega}$, there exists a family $(U_\alpha:\alpha\in D)$ such that:
    \begin{enumerate}[label=(\alph*)]
        \item $U_\alpha \in p_\alpha$ for every $\alpha \in D$,
        \item $U_\alpha \cap U_\beta=\emptyset$ for every distinct $\alpha, \beta \in D$,
        \item $\{h_\alpha(n): n \in U_\alpha, \alpha \in D\}$ is linearly independent.
    \end{enumerate}
\end{proposition}

We refine their construction to obtain a version that allows countably many disjoint blocks of indices, each with an eventually linearly independent family, while preserving disjointness of the selected sets.

We will rely on the following elementary fact about linear algebra.

\begin{lemma}\label{lemma:linearAlgebra2}
Let $A, B$ be linearly independent subsets of a vector space $V$ over a field $F$.
Then there exists $B'\subseteq B$ such that:
\begin{enumerate}[label=(\roman*)]
    \item $|B'|\leq |A|$.
    \item $B\setminus B'$ is disjoint from $A$.
    \item $(B\setminus B') \cup A$ is linearly independent and generates the same subspace as $A\cup B$.
\end{enumerate}
\end{lemma}

\begin{proposition}
    There exists a family of ultrafilters on $\omega$, $(p_\alpha: \alpha\in \mathfrak c)$ such that:
    \begin{enumerate}[label=(\roman*)]
        \item for every set $E$,
        \item for every countable $D\subseteq \mathfrak c$,
        \item for every family $(K_\alpha:\alpha\in D)$ of pairwise disjoint countable subsets of $E$,
        \item for every family $(h_\beta: \beta \in \bigcup_{\alpha \in D} K_\alpha)$ of functions $h_\beta: \omega\to [E]^{<\omega}$ such that, for every $\alpha \in D$, there exists a family $(F_\beta: \beta \in K_\alpha)$ of finite subsets of $\omega$ such that $(h_{\beta}(m):\beta \in K_\alpha, m\in \omega \setminus F_\beta)$ is a linearly independent family in $[E]^{<\omega}$,\label{item:linearlyIndependent}
    \end{enumerate}
    there exists a family $(U_\alpha: \alpha \in D)$ such that:
    \begin{enumerate}[label=(\alph*)]
        \item $U_\alpha \in p_\alpha$ for every $\alpha \in D$,
        \item $U_\alpha \cap U_\beta=\emptyset$ for every distinct $\alpha, \beta \in D$,
        \item there exists a family $(\tilde F_\beta: \beta \in \bigcup_{\alpha \in D} K_\alpha)$ of finite subsets of $\omega$ such that
        
        \begin{equation*}
        (h_{\beta}(m): m \in U_{\alpha_\beta}\setminus \tilde F_\beta,\, \beta \in \bigcup_{\alpha \in D} K_\alpha)
        \end{equation*}

        is linearly independent in $[E]^{<\omega}$, where $\alpha_\beta$ is the unique ordinal such that $\beta \in K_{\alpha_\beta}$.
    \end{enumerate}
\end{proposition}

\begin{proof}
    Let $(I_n: n \in \omega)$ be a partition of $\omega$ into finite sets such that, for every $n \in \omega$, $|I_n|>n^2\sum_{k<n}|I_k|$, and $I_0<I_1<\cdots$.
    Let $\mathcal C=\{B\subseteq \omega:\, |I_n\setminus B|\leq n\sum_{k<n}|I_k|$ for every $n \in \omega\}$.

    \begin{claim}
        $\mathcal C$ is centered.
        Moreover, for every infinite set $R\subseteq \omega$, $\{\bigcup_{n\in R} I_n\}\cup \mathcal C$ is centered.
    \end{claim}
    \begin{proof}[Proof of the claim]Fix $R$.
        Given $B_1, \ldots, B_m\in \mathcal C$, let $B=\bigcap_{i=1}^m B_i$.
        Then, for every $n \in \omega$ with $m\leq n$, we have:

        \begin{align*}
            |I_n\setminus B| & \leq \sum_{i=1}^m |I_n\setminus B_i| \leq m n \sum_{k<n}|I_k| \leq n^2 \sum_{k<n}|I_k|<|I_n|.
        \end{align*}

        Therefore, $|I_n\cap B|>0$ for every $n \in R\setminus m$, so $B\cap \bigcup_{n\in R} I_n$ is infinite.
    \end{proof}
    Let $(a_\alpha: \alpha<\mathfrak c)$ be an almost disjoint family of cardinality $\mathfrak c$, enumerated without repetitions.
    Let $\mathcal F$ be the free filter generated by $\mathcal C$.
    For each $\alpha<\mathfrak c$, let $p_\alpha$ be an ultrafilter extending the filter generated by $\mathcal F|\bigcup_{n\in a_\alpha} I_n=\{B\cap \bigcup_{n\in a_\alpha} I_n:\, B \in \mathcal F\}$.

    We now show that the family $(p_\alpha: \alpha<\mathfrak c)$ satisfies the required properties.

    Let $E$ be a set, let $D\subseteq \mathfrak c$ be countable, and let $(K_\alpha:\alpha\in D)$ be a family of pairwise disjoint, countable subsets of $E$.

    Let $(B_\alpha: \alpha \in D)$ be a partition of $\omega$ such that $a_\alpha\subseteq^* B_\alpha$ for every $\alpha \in D$.

    For every $k \in \omega$, let $\alpha_k$ be the unique ordinal such that $k \in B_{\alpha_k}$.
    Then $(\alpha_k:k\in\omega)$ is a sequence in $D$ in which every element of $D$ appears infinitely many times.
    
    For each $\alpha\in D$, let $(\beta_n^\alpha: n < |K_\alpha|)$ be an enumeration of $K_\alpha$ without repetitions.

    Let $\rho(k)=|\{i<k: \alpha_k=\alpha_i\}|$.
    Then $\rho(k)$ counts how many appearances of $\alpha_k$ have occurred in $(\alpha_i: i \in \omega)$ before stage $k$.

    Let
    \[
    T_k=\{\beta_i^{\alpha_k}: i<\min(\rho(k), |K_{\alpha_k}|)\}.
    \]
    Thus, $T_k$ consists of an initial finite segment of the enumeration of
    $K_{\alpha_k}$.

    Let $(h_\beta: \beta \in \bigcup_{\alpha \in D} K_\alpha)$ and $(F_\beta: \beta \in \bigcup_{\alpha \in D} K_\alpha)$ be as in the hypothesis.

    Now recursively, choose sets $C_n$ for $n \in \omega$ such that:
    \begin{enumerate}[label=(\alph*)]
        \item $C_n \subseteq I_n$,
        \item $|C_n|\leq n \sum_{k<n}|I_k|$,
        \item the family $(h_{\beta}(m): \beta \in T_k, m \in I_k\setminus (C_k\cup F_\beta), k\leq n)$ is linearly independent.
    \end{enumerate}

    To see that this is possible, having defined $C_k$ for $k<n$, the family $(h_{\beta}(m): \beta \in T_k, m \in I_k\setminus (C_k\cup F_\beta), k<n)$ is linearly independent.
    Note that this family has cardinality at most:

    \begin{equation*}
        \sum_{k<n} |T_k| |I_k| \leq \sum_{k<n} k |I_k| \leq n \sum_{k<n} |I_k|.
    \end{equation*}
    The family
    $(h_{\beta}(m):\beta\in T_n,\ m\in I_n\setminus F_\beta)$ is also linearly
    independent. Lemma~\ref{lemma:linearAlgebra2} therefore gives a set
    $C_n\subseteq I_n$ of cardinality at most
    $n\sum_{k<n}|I_k|$ such that
    $(h_{\beta}(m):\beta\in T_k,\ m\in I_k\setminus(C_k\cup F_\beta),\ k\leq n)$
    is linearly independent.

    Finally, let $U_\alpha=\bigcup\{I_n\setminus C_n:\, n \in B_\alpha\}$ for every $\alpha \in D$.
    For each $\beta \in \bigcup_{\alpha \in D} K_\alpha$, choose $i_\beta\in\omega$ such that
    $i_\beta\in B_{\alpha_\beta}$ and $\beta\in T_{i_\beta}$, and let
    \[
    \tilde F_\beta=\Big(\bigcup_{k<i_\beta} I_k\Big)\cup F_\beta.
    \]
    Then:

    \begin{itemize}
        \item $U_\alpha \in p_\alpha$ for every $\alpha \in D$: as $\bigcup_{n \in B_\alpha} I_n \supseteq^* \bigcup_{n \in a_\alpha} I_n$ and the latter is in $p_\alpha$, so is the former.
        Moreover, $\bigcup_{n\in\omega}(I_n\setminus C_n)\in\mathcal F$, so
        \[
        U_\alpha=
        \left(\bigcup_{n\in B_\alpha}I_n\right)
        \cap\left(\bigcup_{n\in\omega}(I_n\setminus C_n)\right)
        \in p_\alpha.
        \]
        \item $U_\alpha \cap U_{\alpha'}=\emptyset$ for every distinct $\alpha, \alpha' \in D$ as $B_\alpha\cap B_{\alpha'}$ is empty.
        \item The family
        \[
        (h_{\beta}(m): m \in U_{\alpha_{\beta}}\setminus\tilde F_{\beta},\, \beta \in \bigcup_{\alpha \in D} K_\alpha)
        \]
        is linearly independent.
        To see this, choose
        $\beta_0,\dots,\beta_l\in\bigcup_{\alpha\in D}K_\alpha$ and
        $m_0,\dots,m_l$ so that
        \[
        m_j\in U_{\alpha_{\beta_j}}\setminus\widetilde F_{\beta_j}
        \quad(j\leq l),
        \qquad
        (\beta_i,m_i)\neq(\beta_j,m_j)
        \quad(i<j\leq l).
        \]
        Let $k_j$ be such that $m_j \in I_{k_j}$ for $j\leq l$.
        As $m_j \notin \tilde F_{\beta_j}$, we have
        $\beta_j \in T_{k_j}$ and
        $m_j \in I_{k_j}\setminus (C_{k_j}\cup F_{\beta_j})$ for $j\leq l$.
        Indeed, $k_j\in B_{\alpha_{\beta_j}}$, $k_j\ge i_{\beta_j}$, and $\beta_j\in T_{i_{\beta_j}}$, hence $\beta_j\in T_{k_j}$.

        Let $N=\max\{k_0,\dots,k_l\}$. For every $j\leq l$, the vector
        $h_{\beta_j}(m_j)$ belongs to the linearly independent family
        \[
        (h_\beta(m):\beta\in T_k,\
        m\in I_k\setminus(C_k\cup F_\beta),\ k\leq N).
        \]
        Hence, the selected vectors are linearly independent.
    \end{itemize}
\end{proof}

We now isolate the two pieces of structure extracted from the previous proposition that will be used in the sequel.
The first is the ultrafilter family itself; the second is the combinatorial configuration of blocks and functions to which the proposition will be applied.

\begin{definition}
    A family of ultrafilters $(p_\alpha:\alpha\in\mathfrak c)$ satisfying the conclusion of the previous proposition is called a \emph{tidy family of ultrafilters}.
\end{definition}
\begin{definition}

    A \emph{tidy system} is a family $(K_\alpha: \alpha \in Z)$ of pairwise disjoint, countable subsets of $\mathfrak c$, where $Z\subseteq \mathfrak c$, along with a family $(h_\beta: \beta \in \bigcup_{\alpha \in Z} K_\alpha)$ of functions
    $h_\beta:\omega\to[\mathfrak c]^{<\omega}$ such that:

    \begin{enumerate}[label=(\roman*)]
        \item for every $\alpha \in Z$, there exists a family $(F_\beta: \beta \in K_\alpha)$ of finite subsets of $\omega$ such that $(h_{\beta}(m): \beta \in K_\alpha, m \in \omega\setminus F_\beta)$ is a linearly independent family in $[\mathfrak c]^{<\omega}$,
        \item for every $\alpha \in Z$ and every $\beta \in K_\alpha$, $\bigcup_{n \in \omega} h_\beta(n)\subseteq \beta$,
        \item $\omega \cap K_\alpha=\emptyset$ for every $\alpha \in Z$.
    \end{enumerate}
\end{definition}

\subsection{Constructing homomorphisms}

In this subsection, we construct homomorphisms into the Boolean group $2=\{0, 1\}$, endowed with addition modulo $2$.

\begin{definition}
    Let $(K_\alpha: \alpha \in Z)$, $(h_\beta: \beta \in \bigcup_{\alpha \in Z} K_\alpha)$ be a tidy system.
    We say that $D\subseteq \mathfrak c$ is \emph{suitably closed} if:

    \begin{enumerate}[label=(\roman*)]
        \item $\omega \subseteq D$;
        \item for every $\alpha \in D\cap Z$, $K_\alpha\subseteq D$;
        \item for every $\alpha \in Z$, if $K_\alpha \cap D\neq \emptyset$, then $\alpha \in D$;
        \item for every $\alpha \in D\cap Z$ and every $\beta \in K_\alpha$, $\bigcup_{n \in \omega} h_\beta(n) \subseteq D$.
    \end{enumerate}
\end{definition}

A standard closing-off argument yields the following.

\begin{lemma}\label{lemma:existenceOfCountableTidySets}
    Let $(K_\alpha: \alpha \in Z)$, $(h_\beta: \beta \in \bigcup_{\alpha \in Z} K_\alpha)$ be a tidy system.
    Then for every countable set $D'\subseteq \mathfrak c$, there exists a countable suitably closed set $D$ such that $D'\subseteq D$.
\end{lemma}

\begin{proof}
    Set $D_0=D'\cup \omega$. Recursively, given $D_n$, define
    \[
    \begin{aligned}
    D_{n+1}={}&D_n
    \cup \bigcup_{\alpha\in D_n\cap Z}K_\alpha
    \cup \{\alpha\in Z:K_\alpha\cap D_n\neq\emptyset\}\\
    &\cup \bigcup\Bigl\{
       \bigcup_{m\in\omega}h_\beta(m):
       \alpha\in D_n\cap Z,\ \beta\in K_\alpha
       \Bigr\}.
    \end{aligned}
    \]
    By the tidiness assumptions, each $D_n$ is countable. Hence, $D=\bigcup_{n\in\omega} D_n$ is countable.
    Clearly $D'\subseteq D$.
    A routine verification from the construction shows that $D$ is suitably closed.
\end{proof}

\begin{definition}
    Let $(p_\alpha: \alpha \in \mathfrak c)$ be a tidy family of ultrafilters, and let $(K_\alpha: \alpha \in Z)$, $(h_\beta: \beta \in \bigcup_{\alpha \in Z} K_\alpha)$ be a tidy system.
    
    Let $D\subseteq \mathfrak c$ be suitably closed.
    A homomorphism $\Phi: [D]^{<\omega} \to 2$ is said to be \emph{tidy} if
    \begin{equation*}
        \forall \alpha \in D \cap Z,\ \forall \beta \in K_\alpha,\ 
        \Phi(\{\beta\}) = p_\alpha\text{-}\lim_{m \in \omega} \Phi(h_\beta(m)).
    \end{equation*}
\end{definition}

\begin{lemma}\label{lemma:extensionOfTidyHomomorphisms}
    In the previous notation, every tidy homomorphism $\Phi: [D]^{<\omega} \to 2$ can be extended to a tidy homomorphism $\Psi: [\mathfrak c]^{<\omega} \to 2$.
\end{lemma}
\begin{proof}
    Let $\Phi: [D]^{<\omega} \to 2$ be a tidy homomorphism.

    Enumerate $\mathfrak c \setminus D$ increasingly as $\{\beta_\xi: \xi < \lambda\}$, where $\lambda$ is the order type of $\mathfrak c \setminus D$.
    We define by transfinite recursion a family $(\Phi_\xi:\xi\leq\lambda)$ such that:
    \begin{enumerate}[label=(\roman*)]
        \item $\dom(\Phi_\xi) = [D\cup \{\beta_\eta: \eta < \xi\}]^{<\omega}$;
        \item $\Phi_{\xi'} \supseteq \Phi_\xi$ whenever $\xi < \xi' \leq \lambda$;
        \item each $\Phi_\xi$ is a homomorphism.
    \end{enumerate}

    Let $\Phi_0=\Phi$. Suppose that $\xi<\lambda$ and that $\Phi_\xi$ has already been defined. Choose $e_\xi \in 2$ as follows:
    \begin{enumerate}[label=(\alph*)]
        \item if $\beta_\xi \in K_\alpha$ for the unique $\alpha \in Z$, let
        \begin{equation*}
            e_\xi = p_\alpha\text{-}\lim_{m\in\omega}\Phi_\xi(h_{\beta_\xi}(m));
        \end{equation*}
        \item otherwise, let $e_\xi$ be arbitrary.
    \end{enumerate}
    Then define $\Phi_{\xi+1}: [D\cup\{\beta_\eta:\eta\leq\xi\}]^{<\omega}\to 2$ by
    \begin{equation*}
        \Phi_{\xi+1}(s)=
        \begin{cases}
            \Phi_\xi(s), & \text{if }\beta_\xi\notin s,\\
            \Phi_\xi(s\setminus\{\beta_\xi\})+e_\xi, & \text{if }\beta_\xi\in s.
        \end{cases}
    \end{equation*}
    This defines a homomorphism extending $\Phi_\xi$.

    If $\beta_\xi \in K_\alpha$, then, by the definition of tidy system, $\bigcup_{m\in\omega} h_{\beta_\xi}(m)\subseteq \beta_\xi$. Since the enumeration is increasing, it follows that
    \begin{equation*}
        \bigcup_{m\in\omega} h_{\beta_\xi}(m)\subseteq D\cup\{\beta_\eta:\eta<\xi\},
    \end{equation*}
    so the above ultrafilter limit is well-defined.

    For every limit ordinal $\delta\leq\lambda$, let $\Phi_\delta=\bigcup_{\xi<\delta}\Phi_\xi$. Then $\Phi_\delta$ is a homomorphism.
    Finally, let $\Psi=\Phi_\lambda$.
    It is straightforward to verify that $\Psi$ is a tidy homomorphism extending $\Phi$.
\end{proof}

\begin{lemma}\label{lemma:existenceOfHomomorphisms}
    Let $(p_\alpha: \alpha \in \mathfrak c)$ be a tidy family of ultrafilters and $(K_\alpha: \alpha \in Z)$, $(h_\beta: \beta \in \bigcup_{\alpha \in Z} K_\alpha)$ be a tidy system.

    Then for every finite set $A \subseteq \mathfrak c$ and every $g:A\to 2$, there exists a tidy homomorphism $\Phi: [\mathfrak c]^{<\omega} \to 2$ such that $\Phi(\{\xi\})=g(\xi)$ for every $\xi \in A$.

    Moreover, we can guarantee that for every $\beta \in A\cap \bigcup_{\alpha \in Z} K_\alpha$ and every $i\in \{0,1\}$,
    \[
    |\{m \in \omega: \Phi(h_\beta(m))=i\}|=\omega.
    \]
\end{lemma}

\begin{proof}
    Let $\bar D$ be a countable suitably closed set such that $A \subseteq \bar D$, as given by Lemma~\ref{lemma:existenceOfCountableTidySets}.
    \[
        D=\bar D\cap Z
        \qquad\text{and}\qquad
        D'=\bar D\cap \bigcup_{\alpha\in Z}K_\alpha.
    \]
    Since $\bar D$ is suitably closed, we have $D'=\bigcup_{\alpha\in D}K_\alpha$.

    There exist pairwise disjoint sets $(U_\alpha:\alpha\in D)$ and finite sets
    $(F_\beta:\beta\in\bigcup_{\alpha\in D}K_\alpha)$ such that
    $U_\alpha\in p_\alpha$ for every $\alpha\in D$ and

    \[
    (h_{\beta}(m):m\in U_{\alpha_\beta}\setminus F_\beta,\ \beta\in D')
    \]
    is linearly independent, where $\alpha_\beta$ is the unique
    $\alpha\in D$ such that $\beta\in K_\alpha$.

Enumerate $\bar D\setminus A$ as $\{\beta_n: n \in \omega\}$.
For each $n\in \omega$, let $E_n'=A\cup \{\beta_i: i<n\}$.

Recall that the notation $x\dot \cup y$ denotes $x\cup y$ and indicates that $x\cap y=\emptyset$.
We inductively define $R_n$ and $S_n$ for each $n\in \omega$ such that:

\begin{enumerate}[label=\roman*)]
    \item $R_n$ is finite and contains $\bigcup_{\beta \in E'_n\cap D'}F_\beta$.
    \item $S_0=\{\{\beta\}: \beta \in A\}\dot \cup \{h_\beta(m): \beta \in A\cap D', m \in U_{\alpha_\beta}\setminus R_0\}$.
    \item $S_{n+1}=S_n$ if $\beta_n \notin D'$ and $\{\beta_n\} \in \langle S_n\rangle$.
    \item $S_{n+1}=S_n\cup\{\{\beta_n\}\}$ if $\beta_n \notin D'$ and $\{\beta_n\} \notin \langle S_n\rangle$.
    \item $S_{n+1}=S_n\dot\cup \{h_{\beta_n}(m): m \in U_{\alpha_{\beta_n}}\setminus R_{n+1}\}$ if $\beta_n \in D'$ and $\{\beta_n\} \in \langle S_n\rangle$.
    \item $S_{n+1}=S_n\dot\cup \{h_{\beta_n}(m): m \in U_{\alpha_{\beta_n}}\setminus R_{n+1}\}\cup\{\{\beta_n\}\}$ if $\beta_n \in D'$ and $\{\beta_n\} \notin \langle S_n\rangle$.
    \item $S_n$ is linearly independent for every $n \in \omega$.
\end{enumerate}

This is straightforward to obtain by recursively choosing the finite sets $R_n$ large enough, using Lemma~\ref{lemma:linearAlgebra2}.

Notice that, in any case, $\{\beta_n\} \in \langle S_{n+1}\rangle$ for every $n \in \omega$.
Thus, $\{\{\beta\}: \beta \in \bar D\}\subseteq \langle \bigcup_{n\in \omega} S_n\rangle$, which implies that $\left\langle \bigcup_{n\in \omega} S_n\right\rangle=[\bar D]^{<\omega}$.

For $\alpha\in D$ such that there exists $\beta \in A\cap K_\alpha$, partition $U_{\alpha}\setminus R_0$ into two disjoint infinite sets $U_{\alpha}'$ and $U_{\alpha}''$ so that $U_{\alpha}'\in p_{\alpha}$.

Now, we define by induction a family of homomorphisms $\Phi_n: \langle S_n\rangle\to 2$ such that, for every $n \in \omega$:

\begin{enumerate}[label=\Roman*)]
    \item $\Phi_0(\{\beta\})=g(\beta)$ for every $\beta \in A$.
    \item $\Phi_0(h_\beta(m))=g(\beta)$ for every $\beta \in A\cap D'$, $m \in U_{\alpha_\beta}'$.
    \item $\Phi_0(h_\beta(m))=1-g(\beta)$ for every $\beta \in A\cap D'$, $m \in U_{\alpha_\beta}''$.
    \item $\Phi_{n+1}$ extends $\Phi_n$.
    \item $\Phi_{n+1}(h_{\beta_n}(m))=\Phi_{n+1}(\{\beta_n\})$ for every $m \in U_{\alpha_{\beta_n}}\setminus R_{n+1}$ if $\beta_n \in D'$.
\end{enumerate}

We start by defining $\Phi_0: \langle S_0\rangle\to 2$ as follows.
We have that $S_0=\{\{\beta\}: \beta \in A\}\cup \{h_\beta(m): \beta \in A\cap D', m \in U_{\alpha_\beta}\setminus R_0\}$ is linearly independent.
Define $\Phi'_0$ on this set of generators so that:
\begin{itemize}
    \item $\Phi_0'(\{\beta\})=g(\beta)$ for every $\beta \in A$,
    \item $\Phi_0'(h_\beta(m))=g(\beta)$ for every $\beta \in A\cap D'$, $m \in U_{\alpha_\beta}'$,
    \item $\Phi_0'(h_\beta(m))=1-g(\beta)$ for every $\beta \in A\cap D'$, $m \in U_{\alpha_\beta}''$.
    \item $\Phi_0'(s)=0$ for every $s\in S_0$ that is not of the above form.
\end{itemize}
Then $\Phi_0'$ extends to a homomorphism $\Phi_0: \langle S_0\rangle\to 2$ satisfying I)-III).

Assume we have defined $\Phi_n$ satisfying I)-V) for some $n\in \omega$. We define $\Phi_{n+1}$ as follows.
First, define $\Phi_{n+1}'$ on $S_{n+1}\cup \dom \Phi_n$.

\begin{itemize}
    \item $\Phi_{n+1}'(s)=\Phi_n(s)$ for every $s\in \dom \Phi_n$,
    \item If $\{\beta_n\}\in S_{n+1}\setminus \dom \Phi_n$, define $\Phi_{n+1}'(\{\beta_n\})=0$.
    \item If $\beta_n \in D'$ and $m \in U_{\alpha_{\beta_n}}\setminus R_{n+1}$, define $\Phi_{n+1}'(h_{\beta_n}(m))=\Phi_{n+1}'(\{\beta_n\})$.
    \item $\Phi_{n+1}'(s)=0$ for every $s\in S_{n+1}\setminus \dom\Phi_n$ that is not of the above form.
\end{itemize}

As $\langle S_n\rangle\cap\langle S_{n+1}\setminus S_n\rangle=\{0\}$, $\Phi_{n+1}'$ extends to a homomorphism $\Phi_{n+1}:\langle S_{n+1}\cup \dom \Phi_n\rangle\to 2$ satisfying V).
Finally, let
\[
\Phi=\bigcup_{n\in \omega}\Phi_n.
\]
Then $\Phi$ is a homomorphism defined on $\left\langle \bigcup_{n\in \omega} S_n\right\rangle=[\bar D]^{<\omega}$.
By construction, $\Phi(\{\beta\})=g(\beta)$ for every $\beta\in A$.

If $\beta\in A\cap D'$, then
$\Phi(h_\beta(m))=\Phi(\{\beta\})$
for every $m\in U'_{\alpha_\beta}$, and
$\Phi(h_\beta(m))\neq \Phi(\{\beta\})$
for every $m\in U''_{\alpha_\beta}$.
Since $U'_{\alpha_\beta}\in p_{\alpha_\beta}$, it follows that
\[
\Phi(\{\beta\})=p_{\alpha_\beta}\text{-}\lim_{m\in\omega}\Phi(h_\beta(m)).
\]
Moreover, both sets
\[
\{m\in\omega:\Phi(h_\beta(m))=0\}
\quad\text{and}\quad
\{m\in\omega:\Phi(h_\beta(m))=1\}
\]
are infinite.
Now let $\beta\in D'\setminus A$, and let $n\in\omega$ be such that $\beta=\beta_n$.
Then
$U_{\alpha_\beta}\setminus R_{n+1}\in p_{\alpha_\beta}$,
and by (V), $\Phi(h_\beta(m))=\Phi(\{\beta\})$ for every $m\in U_{\alpha_\beta}\setminus R_{n+1}$.
Hence,
\[
\Phi(\{\beta\})=p_{\alpha_\beta}\text{-}\lim_{m\in\omega}\Phi(h_\beta(m)).
\]

Therefore, $\Phi$ is a tidy homomorphism. By Lemma~\ref{lemma:extensionOfTidyHomomorphisms}, $\Phi$ extends to a tidy homomorphism $\Psi:[\mathfrak c]^{<\omega}\to 2$, and this $\Psi$ satisfies the conclusion of the lemma.
\end{proof}

\subsection{A \texorpdfstring{$\mathcal B$}{B}-cascade countably compact group without non-trivial convergent sequences}

In this subsection, we fix $\kappa=\mathfrak{c}$ and write $G=G_\kappa=[\mathfrak c]^{<\omega}$.

\begin{theorem}\label{theorem:exampleOfBCascadeCountablyCompactGroup}
    There exists a group topology on $G$ such that:
    \begin{itemize}
        \item $G$ is Hausdorff.
        \item $G$ is $\mathcal B$-cascade countably compact for each barrier $\mathcal B$ on $\omega$.
        \item $G$ has no non-trivial convergent sequences.
        \item $G^\omega$ is countably compact.
    \end{itemize}
\end{theorem}
\begin{proof} We fix:
\begin{enumerate}[label=(\alph*)]
   \item A partition $\mathfrak c\setminus\omega=Z_0\cup Z_1$ such that $|Z_i|=\mathfrak c$ for each $i<2$.

\item A family $(L_\alpha:\alpha\in Z_0\cup Z_1)$ of pairwise disjoint subsets of $\mathfrak c\setminus\omega$ such that:
\begin{enumerate}[label=(\roman*)]
    \item $L_\alpha=\{\alpha\}$ for every $\alpha\in Z_1$;
    \item $L_\alpha$ is cofinal in $\mathfrak c$ for every $\alpha\in Z_0$.
\end{enumerate}
   \item A $Z_0$-enumeration of all triples $(A_\alpha, \mathcal B_\alpha, f_\alpha)$ such that:
   \begin{enumerate}[label=(\roman*)]
    \item $A_\alpha \in [\omega]^\omega$.
    \item $\mathcal B_\alpha$ is a barrier on $A_\alpha$.
    \item $f_\alpha:\, \mathcal B_\alpha \to G$.
    \item Each such triple appears $\mathfrak{c}$ many times in the enumeration.
   \end{enumerate}
   \item By Theorem~\ref{theorem:cascadeFrameExistence}, fix families $(B_\alpha, (g_\beta:\, \beta \in K_\alpha), (x_t^\alpha: t \in \mathcal T_{\mathcal B_\alpha}))$ for $\alpha \in Z_0$ such that:
   \begin{enumerate}[label=(\roman*)]
    \item $K_\alpha \subseteq L_\alpha$ is countable.
    \item $(B_\alpha,(g_\beta:\, \beta \in K_\alpha), (x_t^\alpha: t \in \mathcal T_{\mathcal B_\alpha}))$ is a $\mathcal B_\alpha$-cascade frame for $f_\alpha$.
    \item For every $t \in \mathcal T_{\mathcal B_\alpha}$, $x_t^\alpha \in \langle f_\alpha(b): b \in \mathcal B_\alpha\rangle\ominus [K_\alpha]^{<\omega}$.
   \end{enumerate}

   \item An enumeration $(u_\alpha:\alpha\in Z_1)$ of all injective sequences of linearly independent elements of $G$ such that
    \[
    \bigcup_{n\in\omega}u_\alpha(n)\subseteq \alpha
    \qquad\text{for every }\alpha\in Z_1.
    \]

    \item For every $\alpha\in Z_1$, let
    \[
    K_\alpha=\{\alpha\}
    \qquad\text{and}\qquad
    h_\alpha=u_\alpha.
    \]
\end{enumerate}

For each $\alpha\in Z_0$, let $\phi_\alpha:\omega\to B_\alpha$ be a bijection and $h_\beta=g_\beta \circ \phi_\alpha$ for each $\beta \in K_\alpha$.

Let $(p_\alpha:\alpha<\mathfrak c)$ be a tidy family of ultrafilters.
Then
\[
(K_\alpha:\alpha\in Z_0\cup Z_1),
\qquad
(h_\beta:\beta\in \bigcup_{\alpha\in Z_0\cup Z_1}K_\alpha)
\]
form a tidy system.

By Lemma~\ref{lemma:existenceOfHomomorphisms}, we may fix a family of tidy homomorphisms $\mathcal F$ with the following properties:
\begin{enumerate}[label=(\Roman*)]
\item for every $A \in G\setminus \{0\}$, there exists $\Phi \in \mathcal F$ such that $\Phi(A)=1$;
\item for every $\alpha \in Z_1$, there exists $\Phi \in \mathcal F$ such that $\Phi(u_\alpha(n))=0$ infinitely often and $\Phi(u_\alpha(n))=1$ infinitely often.
\end{enumerate}

Endow $G$ with the initial topology generated by the family of homomorphisms $\mathcal F$.

As $\mathcal F$ is a collection of homomorphisms, $G$ is a topological group.
As $\mathcal F$ separates points (by (I)), $G$ is Hausdorff.

\begin{claim}
    $G$ has no non-trivial convergent sequences.
\end{claim}
Every non-trivial sequence has an injective subsequence, and every injective sequence has a subsequence of linearly independent elements.
Let $(v_n:n\in\omega)$ be such a subsequence.
Since $\bigcup_{n\in\omega} v_n$ is countable and $\cf(\mathfrak c)>\omega$, there exists $\alpha\in Z_1$ such that $\bigcup_{n\in\omega} v_n\subseteq \alpha$.
By \textup{(e)}, we may assume $v_n=u_\alpha(n)$ for every $n\in\omega$.
By (II), there exists $\Phi\in\mathcal F$ such that $\Phi(u_\alpha(n))=0$ infinitely often and $\Phi(u_\alpha(n))=1$ infinitely often.
Thus, $\Phi\circ u_\alpha$ does not converge in $2$, so $u_\alpha$ does not converge in $G$.
Hence, the original sequence does not converge.

\begin{claim}
    $G$ is $\mathcal B$-cascade countably compact for each barrier $\mathcal B$.
\end{claim}

Let such a barrier $\mathcal B$ and $f:\, \mathcal B \to G$ be given.
There exists $\alpha \in Z_0$ such that $(A_\alpha, \mathcal B_\alpha, f_\alpha)=(\omega, \mathcal B, f)$.

For every $\Phi\in \mathcal F$ and every $\beta\in K_\alpha$, since $\Phi$ is a tidy homomorphism, $\Phi(\{\beta\})$ is a $p_\alpha$-limit of $\Phi\circ h_\beta$.
Since the topology on $G$ is the initial topology generated by $\mathcal F$, convergence in $G$ is equivalent to convergence under every homomorphism in $\mathcal F$.
Hence, $\{\beta\}$ is a $p_\alpha$-limit of $h_\beta$ in $G$ for every $\beta\in K_\alpha$.
By Corollary~\ref{corollary:cascadeFrameLimit}, $x_\emptyset^\alpha$ is a $\mathcal B$-cascade $q$-limit of $f$ for some $q\in \omega^*$, showing that $G$ is $\mathcal B$-cascade countably compact.
\begin{claim}
    $G^\omega$ is countably compact.
\end{claim}
This follows from the previous claim and Corollary~\ref{corollary:productOfCountablyCompactGroups}.
\end{proof}

\subsection{A \texorpdfstring{$\mathcal B$}{B}-countably compact group without convergent sequences whose square is not countably compact}\label{section:secondExample}

In this section, we construct a group that is $\mathcal B$-countably compact for every
barrier $\mathcal B$ but whose square is not countably compact. By
Theorem~\ref{theorem:productOfCountablyCompactGroups}, this group is not
$2$-cascade countably compact.

\begin{theorem}\label{theorem:exampleOfBCountablyCompactGroup}
    There exists a group topology on a Boolean group $H$ of cardinality $\mathfrak c$ such that:
    \begin{itemize}
        \item $H$ is Hausdorff.
        \item $H$ has no non-trivial convergent sequences.
        \item $H$ is $\mathcal B$-countably compact for every barrier $\mathcal B$.
        \item $H^2$ is not countably compact (so $H$ is not $2$-cascade countably compact).
    \end{itemize}
\end{theorem}

\begin{proof}
    Let $G=[\mathfrak c]^{<\omega}$.
    We use the same bookkeeping as in the proof of
    Theorem~\ref{theorem:exampleOfBCascadeCountablyCompactGroup}; in particular,
    every triple $(A_\alpha,\mathcal B_\alpha,f_\alpha)$ in the enumeration
    \textup{(c)} appears $\mathfrak c$-many times.

    We enlarge the family $\mathcal F$ of homomorphisms that generates the
    topology on $G$. For each $A\subseteq\omega$, let
    $f_A:[\omega]^{<\omega}\to 2$ be the unique homomorphism such that
    $f_A(\{n\})=1$ if and only if $n\in A$.
    Since $\omega$ is suitably closed, by Lemma~\ref{lemma:extensionOfTidyHomomorphisms}, we can extend $f_A:[\omega]^{<\omega}\to 2$ to a tidy homomorphism $\tilde f_A:G\to 2$.
    Let $\mathcal F$ be as before, but also containing $\tilde f_A$ for every $A \subseteq \omega$.

    Then, as in the proof of Theorem~\ref{theorem:exampleOfBCascadeCountablyCompactGroup}, $G$ is Hausdorff, countably compact, and has no non-trivial convergent sequences.
    Moreover, $G^\omega$ is countably compact.

    To obtain our example, we construct a subgroup $H$ of $G$ for which, in the subspace topology, $H^2$ is not countably compact, but $H$ is countably compact.
    This subgroup will have cardinality $\mathfrak c$.

    The group $H$ will contain $[\omega]^{<\omega}$, and the sequence $((\{2n\}, \{2n+1\}): n \in \omega)$ will have no accumulation point in $H^2$.

    \begin{claim} Assume $p, q \in \omega^*$ are free ultrafilters. Then:
        \begin{enumerate}[label=\roman*)]
            \item If $p \neq q$ and both $p\text{-}\lim \{2n\}$ and $q\text{-}\lim \{2n\}$ exist in $G$, then they are different.
            \item If $p \neq q$ and both $p\text{-}\lim \{2n+1\}$ and $q\text{-}\lim \{2n+1\}$ exist in $G$, then they are different.
            \item If both $p\text{-}\lim \{2n\}$ and $p\text{-}\lim \{2n+1\}$ exist in $G$, then they are different.
            \item If $p \neq q$ and both $p\text{-}\lim \{2n,2n+1\}$ and $q\text{-}\lim \{2n,2n+1\}$ exist in $G$, then they are different.
        \end{enumerate}
    \end{claim}
    \begin{proof}[Proof of the Claim]
        For (i), let $x=p\text{-}\lim \{2n\}$ and $y=q\text{-}\lim \{2n\}$.
        Let $B \in p \setminus q$ and let $A=\{2n: n \in B\}$. Notice that $B'=\omega \setminus B \in q$.
        Then $\tilde f_A(\{2n\})=1$ for every $n \in B$, and $\tilde f_A(\{2n\})=0$ for every $n \in B'$.
        By continuity, $\tilde f_A(x)=1$ and $\tilde f_A(y)=0$, so $x \neq y$.
        The proof of (ii) is similar.

        For (iii), let $x=p\text{-}\lim \{2n\}$ and $y=p\text{-}\lim \{2n+1\}$.
        Let $A$ be the set of even numbers.
        Then $\tilde f_A(\{2n\})=1$ for every $n \in \omega$, and $\tilde f_A(\{2n+1\})=0$ for every $n \in \omega$.
        By continuity, $\tilde f_A(x)=1$ and $\tilde f_A(y)=0$, so $x \neq y$.
        
        For (iv), let $x=p\text{-}\lim \{2n,2n+1\}$ and $y=q\text{-}\lim \{2n,2n+1\}$.
        Let $B \in p \setminus q$ and let
        \[
        A=\{2n:n\in B\}.
        \]
        Notice that $B'=\omega\setminus B \in q$.
        Then $\tilde f_A(\{2n,2n+1\})=1$ for every $n\in B$, and
        $\tilde f_A(\{2n,2n+1\})=0$ for every $n\in B'$.
        By continuity, $\tilde f_A(x)=1$ and $\tilde f_A(y)=0$, so $x\neq y$.
    \end{proof}

    Write $Z_0=\{\mu_\xi:\xi<\mathfrak c\}$ in increasing order.
    Let
    \[
    \mathcal U=
    \{p\in\omega^*: p\text{-}\lim_{n\in\omega}\{2n\}
    \text{ and }p\text{-}\lim_{n\in\omega}\{2n+1\}\text{ exist in }G\}.
    \]
    Now we define $H_\xi$, $\mathcal U_\xi$ and $(\delta_u: u \in \mathcal U_\xi)$ such that the following hold.
    \begin{enumerate}[label=\alph*)]
        \item $H_0=[\omega]^{<\omega}$.
        \item $H_\beta\subseteq H_\xi$ for $\beta<\xi$.
        \item $\mathcal U_\xi=\{p \in \mathcal U: p\text{-}\lim \{2n\} \in H_\xi$ or $p\text{-}\lim \{2n+1\} \in H_\xi$ or $p\text{-}\lim \{2n, 2n+1\} \in H_\xi\}$.
        \item For every $u \in \mathcal U_\xi$, $\delta_u\notin H_\xi$.
        \item For every $u \in \mathcal U_\xi$, $\delta_u = u\text{-}\lim_n \{2n\}$ or $\delta_u = u\text{-}\lim_n \{2n+1\}$.
        \item $|H_\xi|\leq |\xi|+\omega$.
        \item $H_\xi=\bigcup_{\eta<\xi} H_\eta$ if $\xi$ is a limit ordinal.
        \item If $\ran f_{\mu_\xi}\subseteq H_\xi$, then $H_{\xi+1}$ has a $p^{\mathcal B_{\mu_\xi}}$-limit point of $f_{\mu_\xi}$ for some $p \in \omega^*$ containing $A_{\mu_\xi}$.
    \end{enumerate}

    We verify this is possible.
    Define $H_0=[\omega]^{<\omega}$.
    Let $\mathcal U_0$ be as in clause~\textup{(c)}.
    Then clauses~\textup{(a)}--\textup{(c)} are satisfied.
    Clauses~\textup{(f)}--\textup{(h)} do not need to be verified at this step.
    To verify clauses~\textup{(d)} and~\textup{(e)}, it suffices to see that $\mathcal U_0=\emptyset$.
    We verify that $(\{2n\}: n \in \omega)$ and $(\{2n, 2n+1\}: n \in \omega)$ have no limit points in $H_0$ (the proof for $(\{2n+1\}: n \in \omega)$ is similar).
    Given $N \in \omega$, let $A_N=\{2n: n \geq N\}$.
    Then $\tilde f_{A_N}(\{2n\})=\tilde f_{A_N}(\{2n,2n+1\})=1$ for every
    $n\geq N$. Hence, every limit point $x$ of $(\{2n\}:n\in\omega)$ or
    $(\{2n,2n+1\}:n\in\omega)$ must satisfy $\tilde f_{A_N}(x)=1$ for every
    $N\in\omega$.
    However, if $x\in H_0=[\omega]^{<\omega}$, choose $N\in\omega$ such that
    $x\cap [2N,\infty)=\emptyset$.
    Then $\tilde f_{A_N}(x)=0$, so $x$ cannot be a limit point of $(\{2n\}: n \in \omega)$ or of $(\{2n,2n+1\}: n \in \omega)$.

    Now we proceed to the successor step.
    If $\ran f_{\mu_\xi}\not\subseteq H_\xi$, let $H_{\xi+1}=H_\xi$ and $\mathcal U_{\xi+1}=\mathcal U_\xi$.
    Then clauses~\textup{(a)}--\textup{(h)} are satisfied.

    So assume $\ran f_{\mu_\xi}\subseteq H_\xi$.
    By items \textup{(i)}, \textup{(ii)}, and \textup{(iv)} of the claim and a simple counting argument, $|\mathcal U_\xi|<\mathfrak c$ whenever $|H_\xi|<\mathfrak c$.
    Thus, there exists $\alpha \in Z_0$ such that $(A_\alpha, \mathcal B_\alpha, f_\alpha)=(A_{\mu_\xi}, \mathcal B_{\mu_\xi}, f_{\mu_\xi})$ and $L_\alpha\cap\left(\bigcup_{u \in \mathcal U_\xi}\delta_u\cup \bigcup H_\xi\right) =\emptyset$.

    If $x_\emptyset^\alpha\in H_\xi$, let $H_{\xi+1}=H_\xi$ and $\mathcal U_{\xi+1}=\mathcal U_\xi$.
    Then \textup{(d)} and \textup{(e)} are preserved, and \textup{(h)} holds because $x_\emptyset^\alpha$ is a $p^{\mathcal B_{\mu_\xi}}$-limit point of $f_{\mu_\xi}$ already in $H_\xi$.

    So assume $x_\emptyset^\alpha\notin H_\xi$.
    Since $x_\emptyset^\alpha\in \langle f_\alpha(b): b\in\mathcal B_\alpha\rangle\ominus [K_\alpha]^{<\omega}$
    and $\{f_\alpha(b): b\in\mathcal B_\alpha\}\subseteq H_\xi$, there exist $h\in H_\xi$ and $s\in [K_\alpha]^{<\omega}$ such that
    $x_\emptyset^\alpha=h\ominus s$.
    As $K_\alpha\subseteq L_\alpha$ and every element of $H_\xi$ is disjoint from $L_\alpha$, it follows that $s\neq\emptyset$, and hence
    $x_\emptyset^\alpha\cap L_\alpha=s\neq\emptyset$.
    In particular, $x_\emptyset^\alpha\neq \delta_u$ for every $u\in\mathcal U_\xi$.

    Let $H_{\xi+1}=H_\xi\ominus \langle x_\emptyset^\alpha \rangle$ and
    $\mathcal U_{\xi+1}$ be as in \textup{(c)}. It remains to verify
    \textup{(d)} and \textup{(e)}.

    First, we check that if $u \in \mathcal U_\xi$, then $\delta_u \notin H_{\xi+1}$.
    Indeed, by the choice of $\alpha$, every $\delta_u$ and every element of $H_\xi$ is disjoint from $L_\alpha$.
    Since $x_\emptyset^\alpha\cap L_\alpha\neq\emptyset$, every element of $H_{\xi+1}\setminus H_\xi$ has nonempty intersection with $L_\alpha$.
    Therefore $\delta_u\notin H_{\xi+1}$ for every $u\in\mathcal U_\xi$.

    By clauses~\textup{(d)} and~\textup{(e)} at stage $\xi$, at least one of the
    following holds for each $u\in\mathcal U_\xi$.
    \begin{enumerate}[label=\Roman*)]
        \item $u\text{-}\lim \{2n\} \in H_{\xi}$.
        \item $u\text{-}\lim \{2n+1\} \in H_{\xi}$.
        \item $u\text{-}\lim \{2n, 2n+1\} \in H_{\xi}$.
    \end{enumerate}

    Cases~\textup{(I)} and~\textup{(II)} cannot hold simultaneously, since we would have $\delta_u\in H_\xi$.
    If Cases~\textup{(I)} and~\textup{(III)} hold, then
    $u\text{-}\lim \{2n+1\}=u\text{-}\lim \{2n\}\ominus
    u\text{-}\lim \{2n,2n+1\}\in H_\xi$, reducing to the preceding
    contradiction. The combination of Cases~\textup{(II)} and~\textup{(III)}
    is similar. Hence exactly one of the three cases holds.

    Thus, for each $u \in \mathcal U_{\xi+1}$, we may define $\delta_u= u-\lim \{2n\}$ or $\delta_u= u-\lim \{2n+1\}$ so that $\delta_u \notin H_{\xi+1}$.
    This completes the proof of d) and e).

    For the limit step, let $H_\xi=\bigcup_{\eta<\xi} H_\eta$ and $\mathcal U_\xi=\bigcup_{\eta<\xi} \mathcal U_\eta$.
    Then clauses~\textup{(a)}--\textup{(h)} follow directly.
    This completes the construction.

    Now let $H=\bigcup_{\xi <\mathfrak c} H_\xi$.
    Since $H\subseteq G=[\mathfrak c]^{<\omega}$, we have $|H|\leq\mathfrak c$.
    Moreover, $H$ is infinite and, as we will see, $H$ is a countably compact topological group.
    Hence $|H|\geq\mathfrak c$, and therefore $|H|=\mathfrak c$.

    Since $H$ is a subgroup of the Hausdorff group $G$ endowed with the subspace topology, $H$ is Hausdorff and has no non-trivial convergent sequences.

    \textbf{$H$ is $\mathcal B$-countably compact for every barrier $\mathcal B$}: let such a barrier $\mathcal B$ on $A$ and $f:\, \mathcal B \to H$ be given.

    Since $\mathcal B$ is a barrier on a countable set, $\mathcal B$ is countable, and therefore $\ran f$ is countable.
    As $H=\bigcup_{\xi<\mathfrak c} H_\xi$, there exists $\eta<\mathfrak c$ such that $\ran f\subseteq H_\eta$.

    By the choice of the enumeration, the triple $(A,\mathcal B,f)$ appears $\mathfrak c$-many times, so we may choose $\xi\geq\eta$ such that $(A_{\mu_\xi},\mathcal B_{\mu_\xi},f_{\mu_\xi})=(A,\mathcal B,f)$.
    Then $\ran f_{\mu_\xi}\subseteq H_\xi$, and by clause \textup{(h)} of the construction, $H_{\xi+1}$ has a $p^{\mathcal B}$-limit point of $f$ for some $p\in\omega^*$ containing $A$.

    \textbf{$H^2$ is not countably compact}: consider the sequence $((\{2n\}, \{2n+1\}): n \in \omega)$ in $H^2$.
    If this sequence had an accumulation point $(x,y)$ in $H^2$, then there would exist $u \in \omega^*$ such that $x=u\text{-}\lim \{2n\}$ and $y=u\text{-}\lim \{2n+1\}$.
    Thus, $u \in \mathcal U$ and $x, y \in H_\xi$ for some $\xi <\mathfrak c$.
    However, $\delta_u=x$ or $\delta_u=y$, and $\delta_u \notin H_\xi$, a contradiction.

\end{proof}

    \section{Cascade-compact powers and Vietoris hyperspaces}\label{section:hyperspaces}
The Vietoris hyperspaces were first studied by Vietoris in \cite{vietoris1922bereiche}.
The Vietoris hyperspace of a topological space $X$ is the space of all non-empty closed subsets of $X$ endowed with a natural topology.

\begin{definition}
Let $(X, \tau)$ be a $T_1$ topological space (that is, a topological space in which singletons are closed).
We denote the collection of all non-empty closed subsets of $X$ by $\exp X$.
The \emph{Vietoris topology} on $\exp X$ is the topology generated by the sub-base $\{U^+, U^-: U \in \tau\}$, where
\begin{equation*}
    U^+ = \{F \in \exp X: F \subseteq U\} \text{ and } U^- = \{F \in \exp X: F \cap U \neq \emptyset\}.
\end{equation*} 
\end{definition}
If $(X, d)$ is a compact metric space, the topology induced by the Hausdorff metric on $\exp X$ coincides with the Vietoris topology \cite{michael1951topologies}.

A classical result, due to Vietoris, states that the hyperspace of a space is compact if and only if the space itself is compact \cite{vietoris1922bereiche}.
Regarding generalizations of compactness, Ginsburg proved that, for a Tychonoff space $X$:
\begin{enumerate}[label=(\alph*)]
    \item \cite[Corollary 2.3]{ginsburg1975some} If all powers of $X$ are countably compact, then $\exp X$ is countably compact.
    \item \cite[Corollaries 2.3 and 2.7]{ginsburg1975some} If $\exp X$ is pseudocompact (respectively, countably compact), then all finite powers of $X$ are pseudocompact (respectively, countably compact).
\end{enumerate}

Ginsburg asked whether $\exp X$ is pseudocompact whenever all powers of $X$ are pseudocompact.
A counterexample was first provided in \cite[Theorem 5.1]{hruvsak2007pseudocompactness}.
The example was strengthened in \cite{ortiz2018small}: $\exp X$ may fail to be
pseudocompact even when every power of $X$ of cardinality less than
$\mathfrak h$ is countably compact. Recall that the \emph{distributivity
number} $\mathfrak h$ is the least cardinality of a family of
$\subseteq^*$-dense, $\subseteq^*$-downward closed subsets of
$[\omega]^{\omega}$ whose intersection is empty.
It is easily seen that $\aleph_1\leq \mathfrak h\leq \mathfrak c$ and it is well-known that both inequalities may be consistently strict.
For more on $\mathfrak h$, see, e.g., \cite{blass2009combinatorial}.

In this section, we further improve the result from \cite{ortiz2018small} by constructing a space $X\subseteq \beta\omega$ such that all powers of $X$ of cardinality less than $\mathfrak h$ are $n$-cascade countably compact for every $n \in \omega$, and yet $\exp X$ is not pseudocompact.

We use the following lemma from \cite{ortiz2018small}.

\begin{lemma}[{\cite[Lemma 2.1]{ortiz2018small}}]\label{lemma:accumulationPointsHyperspace}
    Let $X \subseteq \beta\omega$ with $\omega \subseteq X$, and let $(F_n : n \in \omega)$ be a sequence in $\exp X$ with $F_n \subseteq \omega$.
    Then $(F_n : n \in \omega)$ has a $p$-limit point in $\exp X$ if and only if every sequence in $\prod_{n \in \omega} F_n$ has a $p$-limit point in $X$, and, in that case,
    \[
        p\text{-}\lim(F_n : n \in \omega)=\cl_X\{p\text{-}\lim f : f \in \prod_{n \in \omega} F_n\}.
    \]
\end{lemma}

We first prove a general construction theorem.
\begin{theorem}\label{theorem:mainHyperspace}
    Let $\mu$ and $\lambda$ be two uncountable cardinals such that:
    \begin{enumerate}[label=(\arabic*)]
        \item $\omega_1\leq \mu \leq \mathfrak c \leq \lambda \leq 2^{\mathfrak c}$.
        \item For every $\theta<\lambda$, $\theta^{<\mu}\leq\lambda$.
        \item $\cf(\lambda) \geq \mu$.
        \item For every $Z\subseteq \omega^*$ with $|Z|<\lambda$, every $\kappa<\mu$, and every $n\in\omega$, $(\beta\omega\setminus Z)^\kappa$ is $n$-cascade countably compact.
    \end{enumerate}
Then there exists $X\subseteq \beta\omega$ such that $\omega\subseteq X$, $|X|=\lambda$, and $X^\kappa$ is $n$-cascade countably compact for every $\kappa<\mu$ and every $n\in\omega$, and $\exp X$ is not pseudocompact.
\end{theorem}
\begin{proof}
    Fix a sequence $C=(C_n: n \in \omega)$ of pairwise disjoint subsets of $\omega$ such that $|C_n|=2^n$.
    We will construct $X$ containing $\omega$, so $C_n$ will be a sequence of isolated points in $\exp X$.
    We will guarantee that this sequence has no accumulation point in $\exp X$, thus, $\exp X$ is not pseudocompact.

    For that purpose, let $\mathcal A$ be an almost disjoint family of elements of $\prod_{n\in\omega} C_n$ of cardinality $\mathfrak c$; that is, for every two distinct $f,g\in\mathcal A$, we have $f(n)\neq g(n)$ for all but finitely many $n\in\omega$.
    To guarantee that the sequence $(C_n:n\in\omega)$ has no accumulation point in $\exp X$ (and thus, that $\exp X$ is not pseudocompact), by Lemma~\ref{lemma:accumulationPointsHyperspace}, it is sufficient that, for every $p \in \omega^*$ there exists $g \in \mathcal A$ such that $p\text{-}\lim g \notin X$. For each $p \in \omega^*$, let $Z_p=\{p\text{-}\lim g: g \in \mathcal A\}\subseteq \beta\omega$.
    We will construct $X$ such that $Z_p\not\subseteq X$ for any $p \in \omega^*$.
    Also, we have to guarantee that $X^\kappa$ is $n$-cascade countably compact for every $\kappa<\mu$ and $n \in \omega$.
    To ensure this, we intend to enumerate all functions in $\bigcup\{(X^\kappa)^{[\omega]^n} : \kappa<\mu,\ n\in\omega\}$.
    However, $X$ does not exist yet.
    Thus, we will use a bookkeeping function $b:\lambda\to \lambda\times \lambda$ to keep track of the sequences that appear in each step of the construction, that is, a function $b=(b_0, b_1)$ of $\lambda$ onto $\lambda\times \lambda$ such that for every $\alpha<\lambda$, $b_0(\alpha)\leq \alpha$.

    The idea is to recursively construct $X$ by stages $X_\alpha$, controlling its cardinality and adding accumulation points for $n$-dimensional sequences chosen by the bookkeeping function $b$ at each step.
    At the same time, we construct a set $Y$ of forbidden future points of $X$ to guarantee that $Z_p\not\subseteq X$ for every $p \in \omega^*$.

   Formally, we recursively construct $(X_\alpha: \alpha<\lambda)$, $(f^\alpha_\beta: \beta<\lambda)$, $(Y_\alpha: \alpha<\lambda)$ such that:

    \begin{enumerate}[label=(\alph*)]
        \item $X_0=\omega$ and $Y_0=\emptyset$.
        \item $X_\beta\subsetneq X_\alpha\subseteq \beta\omega$ for every $\beta<\alpha<\lambda$.
        \item $Y_\beta\subseteq Y_\alpha\subseteq \omega^*$ for every $\beta<\alpha<\lambda$.
        \item $|X_{\alpha+1}\setminus X_\alpha|<\mu$ and $|Y_{\alpha+1}\setminus Y_\alpha|<\mu$ for every $\alpha<\lambda$.
        \item $\{f^\alpha_\beta: \beta<\lambda\}=\bigcup\{(X_\alpha^\kappa)^{[\omega]^n}: \kappa<\mu,\ n\in\omega\}$.
        \item If $f^{b_0(\alpha)}_{b_1(\alpha)}:[\omega]^n\to X_{b_0(\alpha)}^\kappa$, then there exists $p\in\omega^*$ such that $f^{b_0(\alpha)}_{b_1(\alpha)}$ has an $n$-cascade $p$-limit point in $(X_{\alpha+1})^\kappa$.
        \item If $\alpha<\lambda$ is limit, then $X_\alpha=\bigcup_{\beta<\alpha}X_\beta$ and $Y_\alpha=\bigcup_{\beta<\alpha}Y_\beta$.
        \item $X_\alpha \cap Y_\alpha=\emptyset$ for every $\alpha<\lambda$.
        \item For every $p\in \omega^*$ and $\alpha<\lambda$, if $Z_p\cap X_\alpha\neq \emptyset$ then $Z_p \cap Y_{\alpha+1}\neq \emptyset$.
        \item $|X_\alpha|<\lambda$ and $|Y_\alpha|<\lambda$.
    \end{enumerate}
    Assume that this is possible and let $X=\bigcup_{\alpha<\lambda} X_\alpha$ and $Y=\bigcup_{\alpha<\lambda} Y_\alpha$.

    Fix $\kappa<\mu$ and $n\in\omega$. We show that $X^\kappa$ is $n$-cascade countably compact.
    Let $f:[\omega]^n\to X^\kappa$.
    Since $\cf(\lambda)>\max\{\omega,\kappa\}$, there exists $\alpha<\lambda$ such that $\operatorname{ran}(f)\subseteq X_\alpha^\kappa$.
    Thus, there exists $\beta<\lambda$ such that $f=f_\beta^\alpha$, and there exists $\gamma<\lambda$ such that $b(\gamma)=(\alpha,\beta)$ (so $\alpha\leq\gamma$).
    By (f), there exists $p\in\omega^*$ such that $f_\beta^\alpha$ has an $n$-cascade $p$-limit point in $(X_{\gamma+1})^\kappa$.
    Hence, it also has one in $X^\kappa$.

    To prove that $\exp X$ is not pseudocompact, let $p\in\omega^*$.
    We show that $Z_p\not\subseteq X$.  Assume instead that
    $Z_p\subseteq X$, and let $\alpha<\lambda$ be the first ordinal such that
    $Z_p\cap X_\alpha\neq\emptyset$.  By (i), there exists
    $y\in Z_p\cap Y_{\alpha+1}\subseteq Y\setminus X$, a contradiction.
    Lemma~\ref{lemma:accumulationPointsHyperspace} now shows that
    $(C_n:n\in\omega)$ has no $p$-limit in $\exp X$ for any $p\in\omega^*$,
    and therefore has no accumulation point.  The observation at the beginning
    of the proof completes the argument.

    Now we show such a construction is possible.

    For the limit step, let $X_\alpha=\bigcup_{\beta<\alpha} X_\beta$, $Y_\alpha=\bigcup_{\beta<\alpha} Y_\beta$ and
    \[
    \{f^\alpha_\beta: \beta<\lambda\}=\bigcup\{(X_\alpha^\kappa)^{[\omega]^n}: \kappa<\mu,\ n\in\omega\}.
    \]
    First suppose that $\mu<\lambda$. Then

    \begin{equation*}
    |X_\alpha|\leq \sum_{\beta<\alpha}|X_{\beta+1}\setminus X_\beta|\leq |\alpha|\mu<\lambda.
    \end{equation*}
    If $\mu=\lambda$, then $\alpha<\mu\leq \cf(\lambda)$, so
    \[
    |X_\alpha|\leq \sum_{\beta<\alpha}|X_{\beta+1}\setminus X_\beta|+\omega<\lambda.
    \]
    Similarly, $|Y_\alpha|<\lambda$.

    Moreover, if $\kappa<\mu$, then
    $|(X_\alpha^\kappa)^\omega|\leq\lambda$ by (2).  Since
    $\mu\leq\lambda$, it follows that
    \[
    \left|\bigcup\left\{(X_\alpha^\kappa)^{[\omega]^n}:
    \kappa<\mu,\ n\in\omega\right\}\right|\leq\lambda,
    \]
    so this union may be enumerated as in (e).

    Clearly, all the items hold for the limit step.
    The zeroth step is similarly easy.
    For the successor step, first we construct $Y_{\alpha+1}$ as follows.

    \begin{proof}[Construction of $Y_{\alpha+1}$]
    First, notice that $|\{(p, g)\in \omega^*\times \mathcal A: p-\lim g\in \left(X_\alpha\setminus \bigcup_{\beta<\alpha}X_\beta\right)\}|<\mu$ as $|X_\alpha\setminus \bigcup_{\beta<\alpha}X_\beta|<\mu$ and $p-\lim g \neq q-\lim g'$ if $p, q \in \omega^*$ are distinct or $g, g'\in \mathcal A$ are distinct (since the remainders of the closures of the ranges of $g$ and $g'$ in $\beta\omega$ are disjoint in case $g\neq g'$ as they are almost disjoint subsets of $\omega$).
    Thus, its projection $A=\{p \in \omega^*: \exists g \in \mathcal A: p-\lim g \in \left(X_\alpha\setminus \bigcup_{\beta<\alpha}X_\beta\right)\}$ is of size $<\mu$.

    Let $B=\{p \in A: Z_p\cap Y_\alpha = \emptyset\}$.
    Fix $p\in B$.
    Since $Z_p\cap Y_\alpha=\emptyset$, by item (i) at all previous stages we must have $Z_p\cap X_\beta=\emptyset$ for every $\beta<\alpha$.
    Hence,
    \[
    Z_p\cap \bigcup_{\beta<\alpha}X_\beta=\emptyset.
    \]
    Therefore, for every $g\in\mathcal A$,
    \[
    p\text{-}\lim g\in X_\alpha
    \quad\Longleftrightarrow\quad
    p\text{-}\lim g\in X_\alpha\setminus \bigcup_{\beta<\alpha}X_\beta.
    \]
    It follows that $|\{g\in\mathcal A: p\text{-}\lim g\in X_\alpha\}|<\mu$.
    Thus, there exists $g_p\in\mathcal A$ such that $p\text{-}\lim g_p\notin X_\alpha$.
    Let
    $Y_{\alpha+1}=\{p\text{-}\lim g_p: p\in B\}\cup Y_\alpha$.
    Since $|B|<\mu$, we have $|Y_{\alpha+1}\setminus Y_\alpha|<\mu$.
    Then $Y_{\alpha+1}\cap X_\alpha=\emptyset$, $Y_{\alpha+1}\subseteq \omega^*$ and (i) holds:

    Given $p \in \omega^*$ such that $Z_p\cap X_\alpha\neq \emptyset$, if $Z_p\cap X_\beta\neq \emptyset$ for some $\beta<\alpha$, then $Z_p\cap Y_\alpha\neq \emptyset$ by (i) at step $\beta+1\leq \alpha$.
    On the other hand, if $Z_p\cap X_\beta=\emptyset$ for every $\beta<\alpha$, then $p \in A$ and $p \in B$, so $p\text{-}\lim g_p \in Z_p\cap Y_{\alpha+1}$.
    \end{proof}

    Now we must construct $X_{\alpha+1}$.
    Let $f=f^{b_0(\alpha)}_{b_1(\alpha)}:[\omega]^n\to X_{b_0(\alpha)}^\kappa$.
    Since $b_0(\alpha)\leq \alpha$, we may regard $f$ as a function into $X_\alpha^\kappa$.
    As $|Y_{\alpha+1}|<\lambda$, by item (4) of the hypotheses, $(\beta\omega\setminus Y_{\alpha+1})^\kappa$ is $n$-cascade countably compact.
    Thus, there exist $p\in\omega^*$ and an $n$-cascade $p$-limit point $x\in(\beta\omega\setminus Y_{\alpha+1})^\kappa$ of $f$.

    Write $x=(x_\xi)_{\xi<\kappa}$ and let $S_\alpha=\{x_\xi:\xi<\kappa\}$.
    Then $|S_\alpha|\leq \kappa<\mu$, and $S_\alpha\cap Y_{\alpha+1}=\emptyset$.
    Choose also a point
    \[
    r_\alpha\in \omega^*\setminus (X_\alpha\cup Y_{\alpha+1}\cup S_\alpha).
    \]
    This is possible because $|X_\alpha\cup Y_{\alpha+1}\cup S_\alpha|<\lambda\leq 2^{\mathfrak c}=|\omega^*|.$
    Define $X_{\alpha+1}=X_\alpha\cup S_\alpha\cup\{r_\alpha\}$.
    Then $x\in (X_{\alpha+1})^\kappa$, so $f$ has an $n$-cascade $p$-limit point in $(X_{\alpha+1})^\kappa$.
    Moreover,
    \[
    X_\alpha\subsetneq X_{\alpha+1},\qquad X_{\alpha+1}\cap Y_{\alpha+1}=\emptyset,\qquad |X_{\alpha+1}\setminus X_\alpha|<\mu.
    \]

    Let $\{f_\beta^{\alpha+1}: \beta<\lambda\}$ be an enumeration of
    $\bigcup\{(X_{\alpha+1}^\kappa)^{[\omega]^n}: \kappa<\mu,\ n\in\omega\}$
    as in item~(e).

    This completes the recursion. Since each of the $\lambda$ stages adds
    fewer than $\mu\leq\lambda$ points, $|X|\leq\lambda$. Moreover, the points
    $(r_\alpha:\alpha<\lambda)$ are pairwise distinct, because
    $r_\alpha\notin X_\alpha$ whereas $r_\beta\in X_\alpha$ whenever
    $\beta<\alpha$. Hence, $|X|\geq\lambda$, and therefore $|X|=\lambda$.
\end{proof}

To apply the preceding theorem with $\mu=\mathfrak h$ and
$\lambda=2^{\mathfrak c}$, it remains to verify condition~(4). We use the
following extension of Proposition~\ref{proposition:betaOmegaSmallDimensional}.

\begin{theorem}\label{theorem:betaOmegaSmallDimensional}
    Let $n \in \omega$, let $A \in[\omega]^\omega$, let $\kappa<\mathfrak h$ be an infinite cardinal, let
    $Y\subseteq \beta \omega$, and let $\mu=\max\{|Y|,\omega,\kappa\}$.
    Assume that $\mu<2^{\mathfrak c}$.

    Let $(f^\alpha)_{\alpha<\kappa}$ be functions
    $f^\alpha:[A]^n \to \beta \omega \setminus Y$.
    Let
    \[
        f=(f^\alpha)_{\alpha<\kappa}:[A]^n\to (\beta\omega\setminus Y)^\kappa.
    \]

    Then, for every $B \in [A]^\omega$, there exist $C \in [B]^\omega$ and
    $E \subseteq \omega^*$ with $|E| \leq \mu$ such that, for every
    $p \in C^* \setminus E$, the restriction $f\restriction [C]^n$
    has a $[C]^n$-cascade $p$-limit point in $(\beta\omega\setminus Y)^\kappa$.

    In particular, if $|Y|<2^{\mathfrak c}$, then $(\beta\omega\setminus Y)^\kappa$ is
    $n$-cascade countably compact for every $\kappa<\mathfrak h$ and every $n\in\omega$.
\end{theorem}

\begin{proof}
    We proceed by induction on $n$.

    If $n=0$, then $[\omega]^0=\{\emptyset\}$, so we may take $C=B$ and
    $E=\emptyset$.

    Assume the statement holds for $n$, and let
    \[
        f=(f^\alpha)_{\alpha<\kappa}:[A]^{n+1}\to (\beta\omega\setminus Y)^\kappa
    \]
    be given.
    For each $\alpha<\kappa$ and each $m\in A$, define
    $f^\alpha_m:[A\setminus (m+1)]^n\to \beta\omega\setminus Y$
    by
    \[
        f^\alpha_m(s)=f^\alpha(\{m\}\cup s)
        \qquad (s\in [A\setminus (m+1)]^n).
    \]
    Also, define
    \[
        f_m=(f^\alpha_m)_{\alpha<\kappa}:[A\setminus (m+1)]^n\to (\beta\omega\setminus Y)^\kappa
    \]
    for each $m\in A$.

    Fix $B\in [A]^\omega$.
    Recursively define a $\subseteq$-decreasing sequence $(B_k)_{k\in\omega}$ in $[B]^\omega$,
    a sequence $(m_k)_{k\in\omega}$ in $B$, and sets $(E_k)_{k\in\omega}$ such that:
    \begin{enumerate}[label=(\alph*)]
        \item $B_0=B$;
        \item $m_k=\min B_k$ for every $k\in\omega$;
        \item $B_{k+1}\in [B_k\setminus (m_k+1)]^\omega$ for every $k\in\omega$;
        \item $E_k\subseteq \omega^*$ and $|E_k|\leq \mu$ for every $k\in\omega$;
        \item for every $p\in B_{k+1}^*\setminus E_k$, the restriction $f_{m_k}\restriction [B_{k+1}]^n$
        has a $[B_{k+1}]^n$-cascade $p$-limit point in $(\beta\omega\setminus Y)^\kappa$.
    \end{enumerate}
    This is possible by the inductive hypothesis.

    Let $\widetilde B=\{m_k:k\in\omega\}$ and $\widetilde E=\bigcup_{k\in\omega}E_k$.
    Since $\mu$ is infinite, $|\widetilde E|\leq\mu$.

    We claim that for every $p\in \widetilde B^*\setminus\widetilde E$,
    every $k\in\omega$, the restriction
    \[
        f_{m_k}\restriction [\widetilde B\setminus (m_k+1)]^n
    \]
    has a $[\widetilde B\setminus (m_k+1)]^n$-cascade $p$-limit point in
    $(\beta\omega\setminus Y)^\kappa$.
    Indeed, $\widetilde B\setminus (m_k+1)\subseteq B_{k+1}$ and
    $\widetilde B\setminus (m_k+1)\in p$. Hence, any
    $[B_{k+1}]^n$-cascade $p$-limit point of
    $f_{m_k}\restriction [B_{k+1}]^n$
    is also a $[\widetilde B\setminus (m_k+1)]^n$-cascade $p$-limit point of
    $f_{m_k}\restriction [\widetilde B\setminus (m_k+1)]^n$.
    This proves the claim.

    For each $\alpha<\kappa$, let $\mathcal D_\alpha$ be the set of all
    $M\in[\widetilde B]^\omega$ such that there exist $D\in[M]^\omega$ and
    $I\subseteq n+1$ satisfying:
    \begin{enumerate}[label=(\roman*)]
        \item $M\setminus D$ is finite;
        \item $f^\alpha[[D]^{n+1}]$ is discrete;
        \item whenever
        $a_0<\cdots<a_n$
        and
        $b_0<\cdots<b_n$
        are elements of $D$, then
        \[
            f^\alpha(\{a_0,\dots,a_n\})=f^\alpha(\{b_0,\dots,b_n\})
            \quad\Longleftrightarrow\quad
            \forall i\in I\ (a_i=b_i).
        \]
    \end{enumerate}

    We claim that each $\mathcal D_\alpha$ is open dense in
    $([\widetilde B]^\omega,\subseteq^*)$.

    To see density, fix $M\in[\widetilde B]^\omega$.
    By Proposition~\ref{prop:discreteImageCarlson}, there exists
    $M'\in[M]^\omega$ such that $f^\alpha[[M']^{n+1}]$ is discrete.
    Applying the Erd\H{o}s--Rado canonical theorem to $f^\alpha\restriction [M']^{n+1}$,
    we obtain $D\in[M']^\omega$ and $I\subseteq n+1$ such that
    \[
        f^\alpha(\{a_0,\dots,a_n\})=f^\alpha(\{b_0,\dots,b_n\})
        \quad\Longleftrightarrow\quad
        \forall i\in I\ (a_i=b_i)
    \]
    whenever
    $a_0<\cdots<a_n$ and $b_0<\cdots<b_n$ are elements of $D$.
    Then $D\in\mathcal D_\alpha$, and $D\subseteq M$, so $\mathcal D_\alpha$ is dense.

    To see openness, let $M\in\mathcal D_\alpha$, witnessed by some
    $D\in[M]^\omega$ and $I\subseteq n+1$, and let $N\in[\widetilde B]^\omega$
    satisfy $N\subseteq^* M$.
    Put $D'=N\cap D$.
    Then $D'\in[N]^\omega$ and $N\setminus D'$ is finite.
    Since $D'\subseteq D$, the set $f^\alpha[[D']^{n+1}]$ is discrete, and the same set $I$ works for $[D']^{n+1}$.
    Hence, $N\in\mathcal D_\alpha$, proving openness.

    Since $\kappa<\mathfrak h$, the intersection
    $\bigcap_{\alpha<\kappa}\mathcal D_\alpha$ is dense. Choose
    \[
        C\in\bigcap_{\alpha<\kappa}\mathcal D_\alpha.
    \]

    For each $\alpha<\kappa$, fix $D_\alpha\in[C]^\omega$ and $I_\alpha\subseteq n+1$
    witnessing that $C\in\mathcal D_\alpha$.
    Then $C\setminus D_\alpha$ is finite,
    $f^\alpha[[D_\alpha]^{n+1}]$ is discrete, and whenever
    $a_0<\cdots<a_n$ and $b_0<\cdots<b_n$ are elements of $D_\alpha$, then
    \[
        f^\alpha(\{a_0,\dots,a_n\})=f^\alpha(\{b_0,\dots,b_n\})
        \quad\Longleftrightarrow\quad
        \forall i\in I_\alpha\ (a_i=b_i).
    \]

    For each $p\in C^*\setminus\widetilde E$ and $c\in C$, choose a point
    $x_c^p\in (\beta\omega\setminus Y)^\kappa$ which is a
    $[\widetilde B\setminus(c+1)]^n$-cascade $p$-limit point of
    $f_c\restriction [\widetilde B\setminus(c+1)]^n$.
    Since $(\beta\omega)^\kappa$ is compact, there exists
    \[
        x^p=p\text{-}\lim_{c\in C}x_c^p \in (\beta\omega)^\kappa.
    \]

    For each $p\in C^*\setminus\widetilde E$, $\alpha<\kappa$, and $c\in D_\alpha$,
    since $x_c^p$ is a $[\widetilde B\setminus(c+1)]^n$-cascade $p$-limit point of
    $f_c\restriction [\widetilde B\setminus(c+1)]^n$, by continuity its $\alpha$-th
    coordinate $\pi_\alpha(x_c^p)$ is a
    $[\widetilde B\setminus(c+1)]^n$-cascade $p$-limit point of
    $f_c^\alpha\restriction [\widetilde B\setminus(c+1)]^n$.
    Since $D_\alpha\setminus(c+1)\subseteq \widetilde B\setminus(c+1)$ and
    $D_\alpha\setminus(c+1)\in p$,
    \[
      \pi_\alpha(x_c^p) \text{ is a } n\text{-cascade } p\text{-limit point of } f_c^\alpha\restriction [D_\alpha\setminus(c+1)]^n.
    \]

    Below, for each $\alpha<\kappa$, we define a set $F_\alpha\subseteq C^*$ with
    $|F_\alpha|\leq\mu$ such that, whenever
    $p\in C^*\setminus(\widetilde E\cup F_\alpha)$, one has
    $\pi_\alpha(x^p)\in \beta\omega\setminus Y$.
    Fix $\alpha<\kappa$.

    \textbf{Case 1: $0\in I_\alpha$.}
    Then for every two distinct $c,d\in D_\alpha$,
    \[
        f^\alpha_c\big[[D_\alpha\setminus(c+1)]^n\big]
        \cap
        f^\alpha_d\big[[D_\alpha\setminus(d+1)]^n\big]
        =\emptyset.
    \]

    We claim that the map $p\longmapsto \pi_\alpha(x^p)$
    is injective on $C^*\setminus\widetilde E$.
    Indeed, let $p,q\in C^*\setminus\widetilde E$ be distinct.
    Since $D_\alpha\in p\cap q$, choose disjoint sets $P\in p$ and $Q\in q$
    with $P,Q\subseteq D_\alpha$.
    Put
    \[
        P_f=\bigcup_{c\in P}f^\alpha_c\big[[D_\alpha\setminus(c+1)]^n\big],
        \qquad
        Q_f=\bigcup_{c\in Q}f^\alpha_c\big[[D_\alpha\setminus(c+1)]^n\big].
    \]
    Since $f^\alpha[[D_\alpha]^{n+1}]$ is discrete in $\beta\omega$, the closures in $\beta\omega$
    of any two disjoint subsets of $f^\alpha[[D_\alpha]^{n+1}]$ are disjoint.
    Hence, $\cl P_f\cap \cl Q_f=\emptyset$.
    For every $c\in P$,
    $\pi_\alpha(x_c^p)\in
        \cl\Bigl(f^\alpha_c[[D_\alpha\setminus(c+1)]^n]\Bigr)\subseteq \cl P_f$.
    Since $P\in p$, $\cl P_f$ is closed and $\pi_\alpha$ is continuous, it follows that
    \[
        \pi_\alpha(x^p)=p\text{-}\lim_{c\in C}\pi_\alpha(x_c^p)\in \cl P_f.
    \]
    Similarly, $\pi_\alpha(x^q)\in \cl Q_f$.
    Hence, $\pi_\alpha(x^p)\neq \pi_\alpha(x^q)$, proving the claim.

    Let
    \[
        F_\alpha=\{p\in C^*\setminus\widetilde E:\pi_\alpha(x^p)\in Y\}.
    \]
    Since $p\mapsto \pi_\alpha(x^p)$ is injective and $|Y|\leq\mu$, we have
    $|F_\alpha|\leq\mu$.

    \textbf{Case 2: $0\notin I_\alpha$.}
    Let $l=\min D_\alpha$.
    For every $c\in D_\alpha\setminus(l+1)$ and every
    $s\in [D_\alpha\setminus(c+1)]^n$, the sets $\{l\}\cup s$ and $\{c\}\cup s$
    agree on all coordinates indexed by $I_\alpha$.
    Hence, $f^\alpha(\{l\}\cup s)=f^\alpha(\{c\}\cup s)$.
    Therefore,
    \[
        f^\alpha_c=
        f^\alpha_l\restriction [D_\alpha\setminus(c+1)]^n
        \qquad\text{for every }c\in D_\alpha\setminus(l+1).
    \]

    Put $F_\alpha=\emptyset$.
    Fix $p\in C^*\setminus\widetilde E$.
    Since $D_\alpha\setminus(l+1)\in p$, let
    \[
        z^\alpha_p\in \beta\omega\setminus Y
    \]
    be a $[D_\alpha\setminus(l+1)]^n$-cascade $p$-limit point of
    $f^\alpha_l\restriction [D_\alpha\setminus(l+1)]^n$.

    For every $c\in D_\alpha\setminus(l+1)$, the point $z^\alpha_p$ is a
    $[D_\alpha\setminus(c+1)]^n$-cascade $p$-limit point of
    $f^\alpha_c\restriction [D_\alpha\setminus(c+1)]^n$.
    Since $\pi_\alpha(x_c^p)$ is also such a point, by
    Corollary~\ref{corollary:uniqueCascadeLimit} we have
    \[
        \pi_\alpha(x_c^p)=z^\alpha_p
        \qquad\text{for every }c\in D_\alpha\setminus(l+1).
    \]

    As $D_\alpha\setminus(l+1)\in p$, it follows that
    $\pi_\alpha(x^p)=p\text{-}\lim_{c\in C}\pi_\alpha(x_c^p)=z^\alpha_p\in\beta\omega\setminus Y$.
    This finishes the second case.

    Finally, let
    \[
        E=\widetilde E\cup\bigcup_{\alpha<\kappa}F_\alpha.
    \]
    Since $\mu\geq\kappa$ and $|F_\alpha|\leq\mu$ for every $\alpha<\kappa$, we have
    $|E|\leq\mu$.

    Fix $p\in C^*\setminus E$.
    Then $\pi_\alpha(x^p)\in \beta\omega\setminus Y$ for every $\alpha<\kappa$, so
    $x^p\in (\beta\omega\setminus Y)^\kappa$.
    For every $c\in C$, since
    \[
        C\setminus(c+1)\subseteq \widetilde B\setminus(c+1)
        \qquad\text{and}\qquad
        C\setminus(c+1)\in p,
    \]
    the point $x_c^p$ is also a $[C\setminus(c+1)]^n$-cascade $p$-limit point of
    $f_c\restriction [C\setminus(c+1)]^n$.

    Since $x^p=p\text{-}\lim_{c\in C}x_c^p$ and $C\in p$, it follows that
    $x^p$ is a $[C]^{n+1}$-cascade $p$-limit point of
    $f\restriction [C]^{n+1}$.
    This completes the induction.

    For the last assertion, assume $|Y|<2^{\mathfrak c}$, fix $\kappa<\mathfrak h$ and $n\in\omega$,
    and let $g:[\omega]^n\to (\beta\omega\setminus Y)^\kappa$.
    Apply the theorem with $A=B=\omega$.
    Recall that $\mu=\max\{|Y|,\omega,\kappa\}<2^{\mathfrak c}$.

    There exist $C\in[\omega]^\omega$ and
    $E\subseteq \omega^*$ with $|E|\leq\mu$ such that every
    $p\in C^*\setminus E$ yields a $[C]^n$-cascade $p$-limit point of
    $g\restriction [C]^n$ in $(\beta\omega\setminus Y)^\kappa$.
    Since $|C^*|=2^{\mathfrak c}>\mu$, we may choose $p\in C^*\setminus E$.
    Thus, $(\beta\omega\setminus Y)^\kappa$ is $n$-cascade countably compact.
\end{proof}

\begin{corollary}\label{corollary:mainHyperspaceH}
    There exists $X\subseteq \beta\omega$ of cardinality $2^{\mathfrak c}$ such that $\omega\subseteq X$, $X^\kappa$ is $n$-cascade countably compact for every $\kappa<\mathfrak h$ and every $n\in\omega$, and $\exp X$ is not pseudocompact. In particular, $X^\omega$ is $n$-cascade countably compact for every $n\in\omega$.
\end{corollary}
\begin{proof}
    Apply Theorem~\ref{theorem:mainHyperspace} with $\mu=\mathfrak h$ and $\lambda=2^{\mathfrak c}$.
    Condition~(1) is immediate.
    For condition~(2), let $\theta<2^{\mathfrak c}$ and $\nu<\mathfrak h$.
    Since $\mathfrak h\leq \mathfrak c$, we have $\nu\leq \mathfrak c$, and therefore
    \[
        \theta^\nu\leq (2^{\mathfrak c})^{\mathfrak c}=2^{\mathfrak c}.
    \]
    Hence, $\theta^{<\mathfrak h}\leq 2^{\mathfrak c}$.
    Condition~(3) follows from König's theorem, and condition~(4) follows from
    Theorem~\ref{theorem:betaOmegaSmallDimensional}.
\end{proof}

By replacing $\mathfrak h$ with $\mathfrak p$, the \emph{pseudointersection number}, in the previous corollary, we obtain a space of smaller cardinality.

\begin{corollary}\label{corollary:mainHyperspaceP}
    There exists $X\subseteq \beta\omega$ of cardinality $\mathfrak c$ such that
    $\omega\subseteq X$, $X^\kappa$ is $n$-cascade countably compact for every
    $\kappa<\mathfrak p$ and every $n\in\omega$, and $\exp X$ is not pseudocompact.
    In particular, $X^\omega$ is $n$-cascade countably compact for every $n\in\omega$.
\end{corollary}
\begin{proof}
    Apply Theorem~\ref{theorem:mainHyperspace} with $\mu=\mathfrak p$ and $\lambda=\mathfrak c$.
    Condition~(1) is immediate.

    Condition~(2) follows from \cite[III.l.26]{kunen2011set}.

    Condition~(3) holds as $\mathfrak p$ is regular (see e.g. \cite{blass2009combinatorial} or \cite{kunen2011set}).

    Condition~(4) follows from Theorem~\ref{theorem:betaOmegaSmallDimensional} as $\mathfrak p\leq \mathfrak h$.

    Therefore, there exists $X\subseteq \beta\omega$ of cardinality $\mathfrak c$ such that
    $\omega\subseteq X$, $X^\kappa$ is $n$-cascade countably compact for every
    $\kappa<\mathfrak p$ and every $n\in\omega$, and $\exp X$ is not pseudocompact.
\end{proof}

Thus, finite cascade countable compactness of every power below $\mathfrak h$
does not imply pseudocompactness of the Vietoris hyperspace.

    \section{Acknowledgements}
The first author was financed, in part, by the São Paulo Research Foundation (FAPESP), Brazil.
Process Number 2025/07302-0.

\subsection*{AI disclaimer}
The authors employed artificial intelligence tools for editorial purposes, including suggestions on grammar, clarity, organization, and proofreading.
The model used was OpenAI's ChatGPT 5.4 Thinking.

All suggestions made by the model were carefully reviewed and edited by the authors.
All definitions, statements, proofs, and final decisions were made and verified by the authors, who take full responsibility for the final version of the manuscript.

    \bibliographystyle{amsplain}
    \bibliography{includes/bibliography}
\end{document}